\documentclass[11pt,reqno,a4paper]{amsart}

\pdfoutput=1

\usepackage[british]{babel}

\usepackage{amssymb,amstext,amsthm,eucal,float,upgreek}
\usepackage{latexsym,amsmath,amsfonts,amscd,mathrsfs,booktabs}

\usepackage{framed}
\usepackage[margin=2.8cm]{geometry}
\usepackage{enumerate}
\usepackage{lscape}
\usepackage[all,cmtip]{xy}
\usepackage[percent]{overpic}
\usepackage{CR5D-macros}
\usepackage{times}

\date{\today}

\begin{document}

\title[Simply-transitive Levi non-degenerate hypersurfaces in $\bbC^3$]
{Classification of simply-transitive
\\ 
Levi non-degenerate hypersurfaces in $\bbC^3$}

\author{Boris Doubrov}
\address{Faculty of Mathematics and Mechanics, 
Belarusian State University, Nezavisimosti avenue
220050, Minsk, Belarus}
\email{doubrov@bsu.by}

\author{Jo\"el Merker}
\address{Laboratoire de Math\'ematiques d'Orsay,
CNRS, Universit\'e Paris-Saclay, 91405 Orsay Cedex, France}
\email{joel.merker@universite-paris-saclay.fr}

\author{Dennis The}
\address{Department of Mathematics and Statistics, 
UiT The Arctic University of Norway, 9037 Troms\o, Norway}
\email{dennis.the@uit.no}

\makeatletter
\@namedef{subjclassname@2020}{%
  \textup{2020} Mathematics Subject Classification}
\makeatother


\begin{abstract}
Holomorphically homogeneous CR real hypersurfaces $M^3 \subset \mathbb{C}^2$ were classified by \'Elie Cartan in 1932.  In the next dimension, we complete the classification of simply-transitive Levi non-degenerate hypersurfaces $M^5 \subset \mathbb{C}^3$ using a novel Lie algebraic approach independent of any earlier classifications of abstract Lie algebras.  Central to our approach is a new coordinate-free formula for the fundamental (complexified) quartic tensor.  Our final result has a unique (Levi-indefinite) non-tubular model, for which we demonstrate geometric relations to planar equi-affine geometry.
\end{abstract}

\maketitle

 \section{Introduction}
 \label{S:intro}

In general CR dimension $n \geqslant 1$, the classification of {\em locally homogeneous} real hypersurfaces $M^{2n+1} \subset \C^{n+1}$ (up to local biholomorphisms) is a vast, infinite problem.  In 1932, \'Elie Cartan~{\cite{Cartan-1932-I, Cartan-1932-II}} settled the $n=1$ case, and substantial efforts have been made over the last 20 years to complete the $n=2$ case, cf.~{\cite{Loboda-2000, Loboda-2001, Loboda-2003, Fels-Kaup-2008, CRhom}}.  Most recently, the remaining ``simply-transitive Levi-nondegenerate'' part of the classification was addressed in~{\cite{KL2019, AtL2019, AkL2020, Loboda-2020}} using normal form methods.  The main goal of this article is to unify and complete this final study through a novel approach.  
Our Theorem \ref{thm-main} presents the final classification, which thereby concludes the $n=2$ case.

Local Lie groups are analytic, so homogeneous $M^{2n+1} \subset \C^{n+1}$ may be assumed from the outset to be real analytic ($\mathcal{C}^\omega$). By Lie's
infinitesimalization principle~{\cite{Lie-2015}}, the group $\Hol(M)$
of local biholomorphic transformations of $\C^{n+1}$ stabilizing $M$
is better viewed as the {\em real} Lie algebra:
\begin{align} \label{E:holM}
\hol(M)
\,:=\,
\Big\{
X
=
{\textstyle{\sum_{k=1}^{n+1}}}\,
a_k(z)\,
\tfrac{\partial}{\partial z_k}
\colon\,\,
\big(
X+\overline{X}
\big)\big\vert_M\,\,
\text{is tangent to}\,\,
M
\Big\},
\end{align}
where $z = (z_1, \dots, z_{n+1})$ are coordinates on $\C^{n+1}$, with the $a_k(z)$ being holomorphic.  As Lie did~{\cite{Lie-2015}}, we will consider {\em local} Lie transformation (pseudo-)groups, and mainly deal with their Lie algebras of vector fields.  Clearly, $M$ is (locally) homogeneous if and only if $\forall p \in M$, the evaluation map $\hol(M) \to T_p M$ sending $X \mapsto (X + \overline{X})|_p$ is surjective.  One calls a homogeneous $M$ {\sl simply-transitive} if $\dim\, M = \dim_\bbR  \hol(M)$, and {\sl multiply-transitive} if $\dim\,M < \dim_\bbR \hol(M)$.

Recall that $M^{2n+1} \subset \C^{n+1}$ is {\sl tubular} (or is a `{\sl tube}') if there is a biholomorphism $M \cong \cS^n \times i\, \R^{n+1}$, where $\cS \subset \R^{n+1}$ is a real hypersurface (its `{\sl base}').  If $\cS = \{ \cF(x_1,\ldots,x_{n+1}) = 0 \} \subset \R^{n+1}$ is a real hypersurface with $d\cF \neq 0$ on $\cS$, its {\sl associated tube} is $M_\cS = \{ \cF \big( \Re z_1, \dots, \Re z_{n+1} \big) = 0 \} \subset \bbC^{n+1}$.  A tube $M_\cS$ is Levi non-degenerate if and only if its base $\cS$ has non-degenerate Hessian, and the signatures of the Levi form and Hessian agree.  Clearly $i\, \partial_{z_1}, \dots, i\, \partial_{z_{n+1}} \in \hol(M_\cS)$. Furthermore, any real affine symmetry $\bS =  \big(A_{k\ell}\,
x_\ell + b_k\big)\, \partial_{x_k}$ (summation assumed on $1 \leq k,\ell \leq n+1$) of $\cS$ has `complexification' $X = \bS^{\CR} = \big(A_{k\ell}\, z_\ell + b_k\big)\, \partial_{z_k}$ in $\hol(M_\cS)$.  Thus, an affinely homogeneous base yields a holomorphically homogeneous tube. 

 \subsection{Main result}
 \label{S:main-result}
 Restrict now considerations to Levi non-degenerate hypersurfaces $M^5 \subset \C^3$, i.e.\ $n=2$.  The multiply-transitive case was tackled in~{\cite{Loboda-2001, Loboda-2003}}, which completed the majority of the classification, except the Levi-indefinite branch with $\dim\, \hol(M) = 6$.  Recently, the entire multiply-transitive classification was settled in~{\cite{CRhom}}.  The simply-transitive case was addressed in~\cite{KL2019,AtL2019,AkL2020,Loboda-2020}, where they employed normal form methods and Mubarakzyanov's classification of 5-dimensional Lie algebras. In this article, we independently settle the entire simply-transitive classification using a novel Lie algebraic approach that does not depend on earlier classifications of abstract Lie algebras.  Our main classification result is\footnote{We use the notation $z_j = x_j + i y_j$ and $w = u+iv$.}:
 
\begin{theorem}
\label{thm-main}
Any simply-transitive Levi non-degenerate hypersurface $M^5 \subset \C^3$ is locally biholomorphic to precisely one of the following.

\smallskip\noindent\textbf{(1)}\,
Either one hypersurface among the $6$ families
of tubular hypersurfaces listed in Table~{\ref{F:tubular-final}}
below, with corresponding $5$ generators of $\hol(M)$.

\smallskip\noindent\textbf{(2)}\,
Or the single {\em non}tubular exceptional model:
\begin{align} \label{E:dmt-saff2}
 \Im(w) \,=\, \big| \Im(z_2) - w\,\Im(z_1) \big|^2,
\end{align}
having indefinite Levi signature and the infinitesimal symmetries:
\begin{align}
z_1\,\partial_{z_1}
-
z_2\,\partial_{z_2}
-
2w\,\partial_w,
\ \ \ \ \
z_1\,\partial_{z_2}
+
\partial_w,
\ \ \ \ \
z_2\,\partial_{z_1}
-
w^2\,\partial_w,
\ \ \ \ \
\partial_{z_1},
\ \ \ \ \
\partial_{z_2},
\end{align}
with Lie algebra structure $\mathfrak{saff} (2,\R) := \mathfrak{sl} (2, \R) \ltimes \R^2$, i.e.\ the planar equi-affine Lie algebra.
\end{theorem}

\begin{footnotesize}
\begin{table}[H]
\[
\begin{array}{|c|@{\,}c@{\,}|@{\,}l@{\,}|c|} 
\hline
\mbox{Label} 
&
\begin{tabular}{c}
Affinely simply-transitive\\ 
non-degenerate real surface $\cF(x_1,x_2,u) = 0$
\end{tabular}
& 
\begin{tabular}{c} 
\rule[0pt]{0pt}{11pt}
Holomorphic symmetries of\\
$\cF(\Re(z_1),\Re(z_2),\Re(w)) = 0$\\ 
beyond $i\partial_{z_1},i\partial_{z_2},i\partial_w$ 
\rule[-5pt]{0pt}{11pt}
\end{tabular} 
& 
\begin{tabular}{c}
Levi-definite\\ 
condition\\
\end{tabular}
\\
\hline\hline
\mathsf{T}1 
& 
\begin{array}{c} 
\rule[0pt]{0pt}{10pt}
u = x_1^\alpha x_2^\beta\\\\
\mbox{\scriptsize Non-degeneracy: } 
\alpha \beta (1-\alpha-\beta) \neq 0\\
\mbox{\scriptsize Restriction: } 
(\alpha,\beta) \neq (1,1), (-1,1), (1,-1)\\
\mbox{\scriptsize Redundancy: } 
(\alpha,\beta) \sim (\beta,\alpha) \sim 
\left(\frac{1}{\alpha},-\frac{\beta}{\alpha}\right)
\rule[-6pt]{0pt}{11pt}
\end{array} 
&
\begin{array}{l}
z_1 \partial_{z_1} + \alpha w \partial_w,\\
z_2 \partial_{z_2} + \beta w \partial_w
\end{array} 
& 
\alpha\beta(1-\alpha-\beta) > 0
\\ 
\hline
\mathsf{T}2 
& 
\begin{array}{c} 
\rule[0pt]{0pt}{11pt}
u = \left(x_1^2+x_2^2\right)^\alpha 
\exp\big(\beta \arctan(\frac{x_2}{x_1})\big)\\\\
\mbox{\scriptsize Non-degeneracy: } 
\alpha \neq \frac{1}{2} \,\,\&\,\, (\alpha,\beta)\neq (0,0)\\
\mbox{\scriptsize Restriction: } (\alpha,\beta) \neq (1,0)\\
\mbox{\scriptsize Redundancy: } (\alpha,\beta) \sim (\alpha,-\beta)
\rule[-5pt]{0pt}{11pt}
\end{array} 
& 
\begin{array}{l}
z_1 \partial_{z_1} + z_2 \partial_{z_2} + 2 \alpha w \partial_w,\\
z_2 \partial_{z_1} - z_1 \partial_{z_2} - \beta w \partial_w
\end{array} 
& 
\alpha > \frac{1}{2} 
\\ 
\hline
\mathsf{T}3
& 
\begin{array}{c} 
\rule[0pt]{0pt}{11pt}
u = x_1\left(\alpha\ln(x_1) + \ln(x_2)\right)\\\\
\mbox{\scriptsize Non-degeneracy: } \alpha\neq -1
\rule[-4pt]{0pt}{11pt}
\end{array} 
& 
\begin{array}{l}
z_1 \partial_{z_1} - \alpha z_2 \partial_{z_2} + w \partial_w,\\
z_2 \partial_{z_2} + z_1 \partial_w
\end{array} 
& 
\alpha < -1 
\\ 
\hline
\mathsf{T}4 
& 
\begin{array}{c}
\rule[0pt]{0pt}{12pt}
\big(u-x_1 x_2+\frac{x_1^3}{3}\big)^2 
= 
\alpha\big(x_2-\frac{x_1^2}{2}\big)^3\\\\
\mbox{\scriptsize Non-degeneracy: } \alpha\neq -\frac{8}{9}\\
\mbox{\scriptsize Restriction: } \alpha\neq 0
\rule[-4pt]{0pt}{11pt}
\end{array} 
&
\begin{array}{l}
z_1 \partial_{z_1} + 2 z_2 \partial_{z_2} + 3 w \partial_w,\\
\partial_{z_1} + z_1 \partial_{z_2} + z_2 \partial_w
\end{array} 
& 
\alpha < -\frac{8}{9} 
\\ 
\hline
\mathsf{T}5 
&
\begin{array}{c}
\rule[0pt]{0pt}{11pt}
x_1 u = x_2^2 + \epsilon x_1^\alpha\\\\
\mbox{\scriptsize Non-degeneracy: } \alpha\neq 1,2\\
\mbox{\scriptsize Restriction: } \alpha\neq 0
\rule[-4pt]{0pt}{11pt}
\end{array} 
& 
\begin{array}{l}
z_1 \partial_{z_1} + \frac{\alpha}{2} z_2 \partial_{z_2} 
+ (\alpha-1) w\partial_w,\\
z_1 \partial_{z_2} + 2 z_2 \partial_w
\end{array}
& 
\epsilon (\alpha-1)(\alpha-2) > 0 
\\
\hline
\mathsf{T}6 
& 
\begin{array}{c} 
x_1 u = x_2^2 + \epsilon x_1^2\ln(x_1)\\
\end{array} 
&
\begin{array}{l}
\rule[0pt]{0pt}{10pt}
z_1 \partial_{z_1} + z_2 \partial_{z_2} + (\epsilon z_1+w) \partial_w,\\
z_1 \partial_{z_2} + 2 z_2 \partial_w
\rule[-5pt]{0pt}{11pt}
\end{array} 
& 
\epsilon = +1 
\\
\hline
\end{array}
\]\vspace{-0.5cm}
\caption{\small
All simply-transitive {\em tubes} 
$M^5 \subset \mathbb{C}^3$. 
Parameters $\alpha, \beta \in \mathbb{R}$ 
and $\epsilon = \pm 1$.}
\label{F:tubular-final}
\end{table}
\end{footnotesize}
\vspace{-0.2cm}

We immediately recover that all simply-transitive Levi-definite $M^5 \subset \C^3$ are tubular~{\cite{KL2019}}.

The classification of {\em affinely homogeneous} surfaces $\cS \subset \R^3$ appears in~{\cite{AffHom,EE1999}}.  A tube $M_\cS$ on an affinely multiply-transitive base $\cS$ is holomorphically multiply-transitive, so for the Levi non-degenerate simply-transitive tube classification, we can start from the DKR list \cite{AffHom} and perform the following:\footnote{Family (6) in \cite[Thm.1]{AffHom} contains a typo: it should also include $\alpha=0$, i.e.\ the Cayley surface.}
\begin{enumerate}[(i)]
 \item Remove those surfaces yielding tubes already appearing in the multiply-transitive classification \cite{CRhom}.  (See our Table \ref{F:MT-CR} and Remark \ref{R:lim-case} in \S\ref{S:tubes}.)
 \item Restrict to affinely simply-transitive surfaces that have non-degenerate Hessians.  (This excludes all quadrics, cylinders, and the {\sl Cayley surface} $u = x_1 x_2 - \frac{x_1^3}{3}$, cf. \cite[Prop. in \S3]{AffHom}.)
 \end{enumerate}
 The desired classification is a subset of the resulting {\em candidate} list, which comprises the surfaces in the 2nd column of Table \ref{F:tubular-final}.  The symmetries in the 3rd column confirm that these all have $\dim\,\hol(M) \geq 5$, but it is important to carefully identify all exceptions for which this dimension jumps up.  Theorem \ref{thm-main} asserts that no such exceptions occur among the candidate list.

A comparison with the simply-transitive list in \cite[Table 7, p.~50]{Loboda-2020} is in order.  The tubular classification there mostly matches ours, but differs in the $\sfT3$ and $\sfT4$ cases in our Table \ref{F:tubular-final}.  For the former, $\alpha=0$ is incorrectly omitted; for the latter, the restriction should be corrected to $\alpha \neq 0,-\frac{8}{9}$.  Moreover, {\em two} nontubular models are listed:
 \begin{enumerate}[(a)]
 \item $(v - x_2 y_1)^2 + y_1^2 y_2^2 = y_1$, which is equivalent to \eqref{E:dmt-saff2} -- see \S \ref{S:Equiv-Loboda}.  We moreover derive \eqref{E:dmt-saff2} in an elementary manner and elucidate some related planar equi-affine geometry.
 \item $v(1+\epsilon x_2 y_2) = y_1 y_2$ with $\epsilon = \pm 1$, which is Levi-degenerate at the origin and Levi-indefinite.  We confirm that $\dim\,\hol(M) = 5$, with generators
\begin{align} \label{E:NT-syms}
\big(2\,i+\epsilon\,z_2^2\big)\,
\partial_{z_1}
+
2\,z_2\,\partial_w,
\ \ \ \ \ 
\epsilon\,w\,\partial_{z_1}
+
\partial_{z_2},
\ \ \ \ \ 
z_1\,\partial_{z_1}
+
w\,\partial_w,
\ \ \ \ \
\partial_{z_1},
\ \ \ \ \ 
 \partial_w.
\end{align}
From the hypersurface equation, $y_2 = \Im(z_2)$ is locally unrestricted, but its level sets are clearly preserved by all symmetries \eqref{E:NT-syms}, so this model is {\em not} homogeneous.
 \end{enumerate}
 
More broadly, Theorem \ref{thm-main} also terminates the  problem of classifying all holomorphically homogeneous CR real hypersurfaces $M^5 \subset \bbC^3$, as follows:
 \begin{enumerate}
 \item {\sl holomorphically degenerate}\footnote{When there exists a nonzero holomorphic vector field $X$ (not only $2\, \Re  X$) that is tangent to $M^{2n+1} \subset \C^{n+1}$, one says that $M$ is {\sl holomorphically degenerate}~{\cite{Merker-Porten-2006,
Merker-2008}}.  After rectifying so that $X =
\partial_{z_{n+1}}$ locally near any $p \in M$ at which
$X\big\vert_p \neq 0$, one locally has $M^{2n+1} \cong \cM^{2n-1}
\times \C$ for some real hypersurface
$\cM^{2n-1} \subset \C^n$. In this case, given any holomorphic function $f(z)$, we have $f(z) \partial_{z_{n+1}} \in \hol(M)$, whence $\dim\, \hol(M) = \infty$.
 }: either the {\sl Levi-flat hyperplane} $\R \times \C \times \C$, or $\cM^3 \times \C$ for some  homogeneous Levi non-degenerate hypersurface $\cM^3 \subset \C^2$, classified by Cartan ~{\cite{Cartan-1932-I, Cartan-1932-II}}. These all have $\dim\,\hol(M) = \infty$.
 \item {\sl holomorphically non-degenerate}: From~{\cite{Merker-Porten-2006}}, there are two possibilities:
 \begin{enumerate}
 \item {\sl constant Levi rank 1 and 2-nondegenerate}: 
 The classification was completed by Fels--Kaup in~{\cite{Fels-Kaup-2008}}.  All such models are tubular, with $\dim\,\hol(M) \leq 10$, which is sharp on the tube with base the future light cone $\cS = \{ x \in \R^3 : x_1^2 + x_2^2 = x_3^2,\, x_3 > 0 \}$.
 \item {\sl Levi non-degenerate}: $\dim\,\hol(M) \leq 15$, which is sharp on the {\sl flat} model $\Im\, w = |z_1|^2 + \epsilon|z_2|^2$, where $\epsilon = \pm 1$.  The biholomorphism $(z_1,z_2,w) \mapsto (z_1,z_2, i(2w - z_1^2 - \epsilon z_2^2))$ maps this to the tube over $u = x_1^2 + \epsilon x_2^2$.
 \end{enumerate}
 \end{enumerate}

 \subsection{Classification approach and further results}
 \label{S:approach}
 Some recent classification approaches focus on effective use of {\em normal forms}.  For instance, in the {\em simply-transitive, Levi-definite} case \cite{KL2019}, the authors realize 5-dimensional real Lie algebras acting transitively on real hypersurfaces by holomorphic vector fields and then find appropriate normal forms for such realizations.  Their starting point is the classification of abstract 5-dimensional {\em real} Lie algebras (Mubarakzyanov \cite{mub1966}), but they also use an important discarding sieve: If $\hol(M)$ is 5-dimensional and contains a 3-dimensional abelian ideal, then $M$ is tubular over an affinely homogeneous base \cite[Prop.3.1]{KL2019}.     In the end, no nontubular models survive and they invoke the DKR classification~{\cite{AffHom}} for tubular cases.

 \begin{remark} By our Theorem \ref{thm-main}, we can {\em a posteriori} assert that \cite[Prop.3.1]{KL2019}, valid for a Lie algebra $\fg$ of holomorphic vector fields acting locally simply transitively on Levi-definite $M^5 \subset \bbC^3$, also holds in the Levi-indefinite case.  However, their proof does not carry over: it relies on \cite[Prop.2.3]{KL2019}, which states that if $X,Y,Z \in \fg$ commute and are linearly independent over $\bbR$ at $q \in M$, then $X,Y,Z$ are linearly independent over $\bbC$ at $q$.  This may fail in the indefinite setting, as the following counterexample shows.  Consider a hypersurface of {\sl Winkelmann type} \cite{CRhom}   given by $\Im(w + \overline{z}_1 z_2) = (z_1)^\alpha (\overline{z}_1)^{\overline\alpha}$ for $\alpha \in \C \backslash \{ -1,0,1,2 \}$, which is tubular if and only if $\frac{(2\alpha-1)^2}{(\alpha+1)(\alpha-2)} \in \R$.  Then $\hol(M)$ contains the abelian subalgebra
 \begin{align}
 X_1 = z_1\partial_{z_2}, \quad X_2 = \partial_{z_2} + z_1 \partial_w, \quad X_3 = i\partial_{z_2} - iz_1 \partial_w, \quad X_4 = \partial_w.
 \end{align}
 Evaluating at a point where $z_1 \neq 0$, we see that $\{ X_1, X_2, X_3 \}$ are linearly independent over $\R$, but they are linearly dependent over $\C$.
 \end{remark}
  
 Our approach to the {\em non-tubular}, simply-transitive classification is substantially different.  Our approach circumvents the use of normal forms, is independent of the Mubarakzyanov classification, and draws upon the known close geometric relationship with so-called {\sl Legendrian contact structures} that was similarly effectively used in \cite{ILC, CRhom}.  (The Cartan-geometric approach \cite{ILC} in the simply-transitive setting would result in heavy case-branching, so this will not be used.)  To describe our strategy, we need to recall some notions.  
  
 Any Levi non-degenerate hypersurface $M^{2n+1} \subset \C^n$ naturally inherits a CR structure of codimension 1, i.e.\ a contact distribution $C = TM \cap J(TM) \subset TM$ with a complex structure $J : C \to C$ compatible with the natural (conformal) symplectic form on $C$.  The induced $J$ on the complexification $C^\C$ has $\pm i$ eigenspaces yielding isotropic, integrable subdistributions.  Abstract CR structures $(M;C,J)$ (for which integrability is not required) have corresponding complexified analogues called {\sl Legendrian contact} (LC) structures $(N;E,F)$.  This consists of a {\em complex} contact manifold $(N^{2n+1},C)$ with the contact distribution $C$ split (instead of $C^\C$) into a pair of isotropic subdistributions $E$ and $F$ of equal dimension.  It is an {\sl integrable} ({\sl ILC}) structure if both $E$ and $F$ are integrable.
 
 Concretely, if $M^{2n+1} \subset \C^{n+1}$ has defining equation $\Phi(z,\overline{z}) = 0$, where $\Phi$ is real analytic, then we define its {\sl complexification} $M^c \subset \C^{n+1} \times \C^{n+1}$ by $\Phi(z,a) = 0$.  (We can recover $M$ as the fixed-point set of the anti-involution $(z,a) \mapsto (\overline{a},\overline{z})$ restricted to $M^c$.)  The associated double fibration
 \begin{align} \label{E:2-fib} \raisebox{0.25in}{
 \xymatrix{ &M^c \ar@{->}[ld]_{\pi_1} \ar@{->}[rd]^{\pi_2}\\
 \C^{n+1} & & \C^{n+1}}}
 \end{align}
 defined by $\pi_1(z,a) = z$ and $\pi_2(z,a) = a$ for $(z,a) \in M^c$, induces vertical (hence integrable) subdistributions $F = \ker(d\pi_1)$ and $E = \ker(d\pi_2)$ on $M^c$.   Levi non-degeneracy of $M$ implies that $C = E \op F$ is a contact distribution on $M^c$, and indeed $(M^c;E,F)$ is an ILC structure. Regarding $a \in \C^{n+1}$ as parameters, we view $M^c = \{ \Phi(z,a) = 0 \}$ as describing a parametrized family of hypersurfaces in $\C^{n+1}$.  These {\sl Segre varieties} were introduced by Segre \cite{Segre1931a, Segre1931b}, further explored by Cartan \cite{Cartan-1932-I} in the $\C^2$ case, and extended more generally -- see for example  \cite{Suk2001, Suk2003, Merker-2008, Merker-Porten-2006, CRhom}.
 
 Locally solving $\Phi(z,a) = 0$ for one variable among $z = (z_1,\ldots,z_{n+1})$, say $w := z_{n+1}$, then differentiating once, we can locally resolve all parameters $a$ in terms of the 1-jet $(z_k, w, w_\ell:= \frac{\partial w}{\partial z_\ell})$ for $1 \leq k,\ell \leq n$.  Hence, we can differentiate one more time, eliminate parameters $a$, and write second partials as a complete 2nd order PDE system (considered up to local {\sl point} transformations):
 \begin{align} \label{E:2-PDE}
 \frac{\partial^2 w}{\partial z_i \partial z_j} = f_{ij}(z_k,w,w_\ell).
 \end{align}
 The Segre varieties are now interpreted as the {\sl space of solutions} of \eqref{E:2-PDE}. (See \eqref{E:PDE-EF} for $E$ and $F$.)
 
 The symmetry algebra of an LC structure consists of all vector fields respectively preserving $E$ and $F$ under the Lie derivative.  In terms of $M^c = \{ \Phi(z,a) = 0 \}$, any symmetry is of the form $X = \xi^k(z) \partial_{z_k} + \sigma^k(a) \partial_{a_k}$.  For example, given a tube $M_\cS = \{ \cF(\Re z) = 0 \}$, its complexification $M_\cS^c = \{ \cF(\frac{z+a}{2}) = 0 \}$ admits the $(n+1)$-dimensional abelian subalgebra $\fa = \langle \partial_{z_1} - \partial_{a_1},\ldots,\partial_{z_{n+1}} - \partial_{a_{n+1}} \rangle$ that is clearly transverse to $E$ and $F$.  In the PDE picture, any  symmetry of \eqref{E:2-PDE} is projectable over the $(z_k,w)$-space, and these are called {\sl point} symmetries.  For Levi non-degenerate $M \subset \C^{n+1}$, the symmetry algebra $\sym(M^c)$ of the associated ILC structure $(M^c;E,F)$ is simply $\hol(M) \otimes_\R \C$, see \cite[Cor. 6.36]{Merker-2008}.  In particular, 
 \begin{align} \label{E:sym-dim}
 \dim_\bbC \sym(M^c) = \dim_\R \hol(M).
 \end{align}
 For our simply-transitive study, $M$ or $M^c$ will be (locally) real or complex Lie groups respectively, and we encode data on their Lie algebras.  Our focus will be on {\sl ASD-ILC triples}:
  
 \begin{defn} \label{D:ILC-triple} Let $\fg$ be a 5-dimensional complex Lie algebra.  An {\sl ILC triple} $(\fg;\fe,\ff) $ consists of a pair of 2-dimensional subalgebras $\fe, \ff$ of $\fg$ with $\fe \cap \ff = 0$ such that for $C := \fe \op \ff$, the map $\eta : \bigwedge^2 C \to \fg / C$ given by $(x,y) \mapsto [x,y] \mod C$ is non-degenerate.  An ILC triple is:
 \begin{enumerate}[(a)]
 \item {\sl tubular} if there exists a 3-dimensional subalgebra $\fa \subset \fg$ with $\fe \cap \fa = \ff \cap \fa = 0$;
 \item {\sl anti-self-dual (ASD)} if there exists an anti-involution $\tau$ of $\fg$ that swaps $\fe$ and $\ff$. In this case, call $\tau$ {\sl admissible}.  In the tubular case, $\tau$ is also required to stabilize $\fa$ above.
 \end{enumerate} 
 \end{defn}

 Given an ASD-ILC triple $(\fg;\fe,\ff)$, the fixed-point set of an admissible anti-involution $\tau$ determines the corresponding Lie algebraic CR data (and conversely). Letting $G$ be a (complex) Lie group with Lie algebra $\fg$, and $E,F$ determined from $\fe,\ff$ by left translations in $G$, the ILC structure $(G;E,F)$ certainly has ILC symmetry dimension, denoted $\dim \ILCsym(\fg;\fe,\ff)$, at least $\dim\, G = 5$.  It is important to recognize and discard cases where it exceeds this.  This occurs when there is an {\sl embedding} (Definition \ref{D:embedding}) into an ILC quadruple $(\tilde\fg,\tilde\fk;\tilde\fe,\tilde\ff)$ with $\dim(\tilde\fk) > 0$.  An important tool in this study is the {\sl fundamental quartic tensor} $\cQ_4$, which we now present.
  
 For any (integrable) CR or ILC structure, it is well-known that there is a fundamental tensor that obstructs local equivalence to the {\sl flat} model, which uniquely realizes the maximal symmetry dimension.  When $n=2$, this tensor takes the form of a {\sl binary quartic} $\cQ_4$, and symmetry upper bounds based on its root type are known -- see \eqref{E:RootTypeSym}.  In the CR setting, $\cQ_4$ is typically computed from the fourth degree part of the Chern--Moser normal form \cite{ELS1999}, while in the SILC setting \cite{ILC} it was computed in terms of a PDE realization \eqref{E:2-PDE}.  However, neither of these methods are amenable to a Lie algebraic approach.  In \S\ref{S:5ILC}, we give a coordinate-free formula for $\cQ_4$ for general LC structures, which can be directly used on Lie algebraic data -- in particular on an ASD-ILC triple $(\fg;\fe,\ff)$.
 
 Our Lie algebraic study is organized in terms of 3-dimensional abelian {\em ideals}.  In \S\ref{S:NoAb}, we efficiently classify all 5-dimensional {\em complex} Lie algebras {\em without} a 3-dimensional abelian ideal (Proposition \ref{P:NoAb}).  The search for ASD-ILC triples supported on this small list of Lie algebras produces a unique model on $\fg = \mathfrak{saff}(2,\C) := \fsl(2,\C) \ltimes \C^2$, see Theorem \ref{T:NoAb}.
 
 In \S\ref{S:NTAb}, we study ASD-ILC triples $(\fg;\fe,\ff)$ with $\fg$ containing a 3-dimensional abelian ideal $\fa$.  Theorem \ref{T:NTab} shows that if $\dim\,\ILCsym(\fg;\fe,\ff) = 5$, then $\fe \cap \fa = \ff \cap \fa = 0$ and $\fa = \tau(\fa)$ under any admissible anti-involution $\tau$.  These data allow us to a priori conclude (Corollary \ref{C:tubes}) that all models in this branch are tubes on an affinely simply-transitive base.
 
 We then return to CR geometry.  In \S\ref{S:ns21}, we construct the exceptional model \eqref{E:dmt-saff2}, highlight related planar equi-affine geometry, and find corresponding PDE realizations.  Finally in \S\ref{S:tubular}, we treat the tubes for any candidate base arising from the DKR classification.  Table \ref{F:tubular-LC-Q4} summarizes the root types for these tubes, which are deduced from the quartics $\cQ_4$ given in Table \ref{F:tubular-LC-Q4}.  From \eqref{E:RootTypeSym}, when the root type is $\pI{}$ or $\pII{}$, the symmetry dimension upper bound is 5, and such models are automatically simply-transitive.  The root type $\pD{}$ and $\pN{}$ cases are more subtle, and simple-transitivity in these remaining cases are confirmed using two methods: PDE point symmetries (\S \ref{S:PDE-method}) and power series (\S \ref{S:ps-method}).
 
 Beyond our main result, let us emphasize two important results obtained in this article:
 \begin{itemize}
 \item We give a simple geometric interpretation and coordinate-free formula for the fundamental quartic tensor $\cQ_4$ for general 5-dimensional LC structures.
 \item We conceptualize and give an effective method for computing symmetries of {\sl rigid} CR structures, which potentially can be generalized to a much larger class of geometric structures.
\end{itemize}

 \section{Fundamental tensor of 5-dimensional Legendrian contact structures}
 \label{S:5ILC}
 
 Motivated by the complexification $M^c \subset \C^{n+1} \times \C^{n+1}$ of a Levi non-degenerate hypersurface $M \subset \C^{n+1}$, we will exclusively study {\em complex} LC structures in this article (but one can carry out analogous constructions for {\em real} LC structures). Recall that a (complex) contact manifold $(N^{2n+1},C)$ consists of a corank one distribution $C$ with non-degenerate skew-bilinear map $\eta : \Gamma(\bigwedge^2 C) \to \Gamma(TN / C)$ given by $X \wedge Y \mapsto [X,Y] \,\,\mod C$.
 
 \begin{defn}
 A {\sl Legendrian contact (LC) structure} $(N;E,F)$ is a (complex) contact manifold $(N,C)$ equipped with a splitting $C = E \op F$ into maximally $\eta$-isotropic (Legendrian) subdistributions $E$ and $F$.
 \end{defn}
 
 For an LC structure, $[\Gamma(E),\Gamma(E)] \subset \Gamma(C)$ and $[\Gamma(F), \Gamma(F)] \subset \Gamma(C)$, so composition with the respective projections provided by the splitting gives two basic structure tensors $\tau_E : \Gamma(\bigwedge^2 E) \to \Gamma(F)$ and $\tau_F : \Gamma(\bigwedge^2 F) \to \Gamma(E)$.  These obstruct the Frobenius-integrability of $E$ and $F$ respectively.  If one of these vanishes, then it is {\sl semi-integrable} ({\sl SILC}), while if both do, then it is {\sl integrable} ({\sl ILC}).  In the SILC case \cite{ILC} with $\tau_F \equiv 0$, there exist local coordinates $(z^k,w,w_k)$ on $N$ such that 
 \begin{align} \label{E:PDE-EF}
 E = \langle \partial_{z^i} + w_i \partial_w + f_{ij} \partial_{w_j} \rangle, \quad F = \langle \partial_{w_i} \rangle,
 \end{align}
 where $f_{ij} = f_{ji}$ are functions of $(z^k,w,w_k)$ and $1 \leq i,j,k \leq n$.  The SILC structure is equivalently encoded by the complete 2nd order PDE system \eqref{E:2-PDE} considered up to local {\em point} transformations, i.e.\ prolongations of transformations of $(z_k,w)$-space.  Compatibility of \eqref{E:2-PDE} is equivalent to $\tau_E \equiv 0$.
 
 Beyond $\tau_E$ and $\tau_F$, there is one additional fundamental tensor $\cW$ that obstructs local equivalence to the {\em flat model} $w_{ij} = 0$.  This curvature was computed for arbitrary $n \geq 2$ in the SILC case \cite[Thm.2.9]{ILC}: with respect to an adapting framing, $\cW$ has components $\cW^{k\ell}_{ij} = \operatorname{trfr}\left( \frac{\partial^2 f_{ij}}{\partial w_k \partial w_\ell} \right)$, symmetric in the upper and lower indices respectively, and where $\trfr$ indicates the completely trace-free part.  When $n=2$, this specializes to a binary quartic tensor field.
 We now revisit the $n=2$ case and derive a coordinate-free formula for $\cW$ for general LC structures.
 
 \subsection{Canonical lifting of a 5-dimensional LC structure}
 \label{S:lift}
 Over $(N^5,C)$, define the $\bbP^1$-bundle $\widetilde{N} \stackrel{\pi}{\to} N$ with fibre over $x \in N$ defined as
 \begin{align}
 \widetilde{N}_x := \{ (\ell_E,\ell_F) \in \bbP(E_x) \times \bbP(F_x) : \eta(\ell_E,\ell_F) = 0 \}
 \end{align}
 Since $\rnk(E) = \rnk(F) = 2$ and $\eta$ restricts to a perfect pairing $E \otimes F \to TN/C$, then $\ell_E$ uniquely determines $\ell_F$, i.e.\ $\ell_F = F \cap (\ell_E)^{\perp_\eta}$, and vice-versa.  Hence, $\widetilde{N} \to N$ is indeed a $\bbP^1$-bundle. The 6-manifold $\widetilde{N}$ is canonically equipped with three distributions $V \subset D \subset \widetilde{C}$:
 \begin{enumerate}[(i)]
 \item rank 1: $V = \ker(\pi_*)$, i.e.\ the vertical distribution for $\pi$;
 \item rank 3: $D|_{\widetilde{x}} := (\pi_*)^{-1}(\ell_E \op \ell_F)$ for $\widetilde{x} = (\ell_E,\ell_F)$;
 \item rank 5: $\widetilde{C} := (\pi_*)^{-1} C$.
 \end{enumerate}
 
 Let us describe these in terms of adapted framings.  Given any $p \in N$, there is always some neighbourhood $U \subset N$ on which we can find a local framing $\{ \be_1, \be_2 , \bf_1, \bf_2 \}$ for $C = E \op F$ with $E = \langle \be_1, \be_2 \rangle$, $F = \langle \bf_1, \bf_2 \rangle$, and structure relations
 \begin{align} \label{E:str-rel}
 [\be_1,\be_2] \equiv 
 [\be_1, \bf_2] \equiv [\be_2,\bf_1] \equiv [\bf_1,\bf_2] \equiv 0, \quad [\be_1,\bf_1] \equiv [\be_2,\bf_2] \not\equiv 0 \quad \mod C.
 \end{align}
 We refer to this as an {\sl LC-adapted framing}. Any such framing induces a local trivialization $\phi\colon \pi^{-1}(U) \to U \times \bbP^1$ of $\widetilde{N} \to N$ via 
 \begin{align}
 \widetilde{x} = (\ell_E|_x,\ell_F|_x) \quad\mapsto\quad (x,[s:t]),
 \end{align}
 where $[s:t]$ are homogeneous coordinates on $\bbP^1$, and 
 \begin{align}
 \ell_E = \langle s\be_1 + t \be_2 \rangle, \quad 
 \ell_F = \langle t\bf_1 - s \bf_2 \rangle.
 \end{align}
 The vector fields $\be_1,\be_2,\bf_1,\bf_2 \in \fX(U)$ naturally induce vector fields on $U \times \bbP^1$ (having trivial component on the $\bbP^1$-factor) and on $\pi^{-1}(U)$ via the trivialization, and {\em we abuse notation to denote these vector fields on $U \times \bbP^1$ or $\pi^{-1}(U)$ also by $\be_1,\be_2,\bf_1,\bf_2$}.  To be explicit, we will work in the local coordinate chart on $\bbP^1$ on which $s \neq 0$, so we may as well assume $s =1$.  Locally we have:
 \begin{align} \label{E:VDC}
 V = \langle \partial_t \rangle, \quad D = \langle \be_1 + t \be_2, t \bf_1 - \bf_2, \partial_t \rangle, \quad \widetilde{C} = \langle \be_1, \be_2, \bf_1, \bf_2, \partial_t \rangle.
 \end{align}
 Using \eqref{E:str-rel}, we confirm that $D$ has weak derived flag $D^{-1} \subset D^{-2} = \widetilde{C} \subset D^{-3} = T\widetilde{N}$ with growth $(\rnk(D^{-1}), \rnk(D^{-2}), \rnk(D^{-3})) = (3,5,6)$.  Moreover, it is straightforward to verify that $(\widetilde{N},D)$ gives an instance of:
 
 \begin{defn}
 A {\sl Borel geometry} $(R^6,D)$ consists of a 6-manifold $R$ equipped with a rank 3 distribution $D \subset TR$ with growth $(3,5,6)$ weak derived flag $D^{-1} := D \subset D^{-2} \subset D^{-3} = TR$ and whose symbol algebra $\fm(x) := D(x) \op (D^{-2}(x)/D(x)) \op (TN/D^{-2}(x))$ at every $x \in R$ is isomorphic (as graded Lie algebras) to $\fm = \fg_{-1} \op \fg_{-2} \op \fg_{-3} = \{ e_1, e_2, e_3\} \op \{ e_4, e_5 \} \op \{ e_6 \}$ satisfying the commutator relations
 \begin{align} \label{E:XXX-symbol}
 [e_1,e_2] = e_4, \quad [e_2,e_3] = e_5, \quad [e_1,e_5] = -e_6, \quad [e_3,e_4]=e_6.
 \end{align}
 \end{defn}
 
 \begin{remark}
 Consider the Borel subalgebra in $\fsl(4)$ consisting of upper triangular trace-free matrices.  There is an induced stratification on the complementary subalgebra of strictly lower triangular matrices and the bracket relations match those for $\fm$ above.  Lifting the LC structure and reinterpreting it as a Borel geometry is an instance of a general construction for parabolic geometries referred to as lifting to a ``correspondence space'' \cite{Cap2005}.  However, we will not need to use any of the broad theory developed there.
 \end{remark}
 
 For any Borel geometry, let us observe that $D$ inherits distinguished subdistributions:
 
 \begin{prop} \label{P:XXX-decomp}
 Given any Borel geometry $(R^6,D)$, we canonically have:
 \begin{enumerate}[(a)]
 \item a rank 2 subdistribution $\sqrt{D} \subset D$ satisfying $[\sqrt{D},\sqrt{D}] \equiv 0 \,\,\,\mod D$;
 \item a line field $V = \{ X \in \Gamma(D) : [X,\Gamma(D^{-2})] \subset \Gamma(D^{-2}) \}$.  This satisfies $D = V \op \sqrt{D}$.
 \item a decomposition $\sqrt{D} = L_1 \op L_2$ (unique up to ordering) into null lines for a canonical (non-degenerate) conformal symmetric bilinear form on $\sqrt{D}$.
 \end{enumerate}
 \end{prop}
 
 \begin{proof} \quad
 \begin{enumerate}[(a)]
 \item The bracket $\bigwedge^2 \fg_{-1} \to \fg_{-2}$ coming from $\bigwedge^2 D \to D^{-2} / D$ has 1-dimensional kernel $\langle e_1 \wedge e_3 \rangle$.  This corresponds to a (rank 2) $\sqrt{D} \subset D$ satisfying $[\sqrt{D},\sqrt{D}] \equiv 0\,\, \mod D$.

 \item The bracket gives a surjective map $\fg_{-1} \times \fg_{-2} \to \fg_{-3}$, so the induced map $\fg_{-1} \to \fg_{-2}^* \ot \fg_{-3}$ has 1-dimensional kernel $\langle e_2 \rangle$.  Thus, there exists a distinguished line field $V \subset D$ satisfying $[X,\Gamma(D^{-2})] \subset \Gamma(D^{-2})$ for any $X \in \Gamma(V)$.  From \eqref{E:XXX-symbol}, it is clear that $V \not\subset \sqrt{D}$.

\item The Lie bracket induces the isomorphism $V \ot \sqrt{D} \cong D^{-2}/D$ and a map $\sqrt{D} \ot (D^{-2}/D) \to TR/D^{-2}$.  Via the former, the latter induces a conformal symmetric bilinear form on $\sqrt{D}$.  In a framing corresponding to the basis $\{ e_1, e_3 \}$, it is a multiple of $\begin{psm} 0 & 1\\ 1 & 0 \end{psm} \mod D^{-2}$. Letting $L_1,L_2 \subset \sqrt{D}$ be complementary null line fields then establishes the claim. \qedhere
 \end{enumerate}
 \end{proof}

The decomposition $D = V \op \sqrt{D}$ provides projections onto each factor.  Consequently, the following result is immediate:

 \begin{cor} \label{C:Phi} The map $\Gamma(L_1) \times \Gamma(L_2) \to \Gamma(V)$ given by\footnote{Because of the possibility of swapping $L_1$ and $L_2$, $\Phi$ is canonical only up to a sign.}
 \begin{align} \label{E:Phi-proj}
 (X,Y) \mapsto \proj_V([X,Y])
 \end{align}
  is tensorial, so determines a vector bundle map $\Phi : L_1 \otimes L_2 \to V$.  Geometrically, it is the obstruction to Frobenius-integrability of $\sqrt{D}$. 
 \end{cor}
 
 For an LC structure $(N^5;E,F)$, we refer to $\Phi$ as its {\sl fundamental tensor}.  We now show that $\Phi$ specializes to the known quartic expression in the SILC case.
 
 \subsection{The fundamental quartic tensor}
 \label{S:framing}
 We now evaluate $\Phi$ in an LC-adapted framing.
 
 \begin{lem} Let $(N^5;E,F)$ be an LC structure, $\{ \be_1, \be_2, \bf_1, \bf_2 \}$ an LC-adapted framing of $C=E \op F$ on $N$ (i.e.\ satisfying \eqref{E:str-rel}) and let $\{ \be^1, \be^2, \bf^1, \bf^2 \}$ be its dual coframing.  Following \S \ref{S:lift}, we induce vector fields on $\widetilde{N}$ satisfying \eqref{E:VDC}.
 \begin{enumerate}
 \item The line fields $V,L_1,L_2$ from Proposition \ref{P:XXX-decomp} are respectively spanned by
 \begin{align} \label{E:L1L2}
 \partial_t, \quad 
 \ell_1 = \be_1 + t \be_2 + A_1\partial_t, \quad
 \ell_2 = t \bf_1 - \bf_2 + A_2\partial_t,
 \end{align}
 where, defining $\bS:= [\be_1 + t \be_2, t \bf_1 - \bf_2]$, we have
 \begin{align} \label{E:SA1A2}
 A_1 = -(\bf^1+t\bf^2)(\bS), \quad A_2 = (\be^2 - t\be^1 )(\bS).
 \end{align}
 \item Defining \framebox{$\cQ_4 := -dt(\Phi(\ell_1,\ell_2))$} in terms of the fundamental tensor $\Phi$, we have
 \begin{equation} \label{E:Phi}
 \cQ_4 = -\ell_1(A_2) + \ell_2(A_1) - \be^1(\bS) \bf^1(\bS) - \be^2(\bS) \bf^2(\bS),
 \end{equation}
 which is a polynomial in $t$ of degree at most 4.
 \end{enumerate}
 \end{lem}
 
 \begin{proof}
 We already know $V = \langle \partial_t \rangle$, so write $\sqrt{D} = \langle \ell_1, \ell_2 \rangle$ with $\ell_1,\ell_2$ as in \eqref{E:L1L2}.  Write
 \begin{align} \label{E:L1L2-bracket}
 [\ell_1,\ell_2] &= \bS +  A_1 \bf_1 - A_2 \be_2 + (\ell_1(A_2) - \ell_2(A_1)) \partial_t,
 \end{align}
 where $\bS \in \Gamma(\widetilde{C})$ by \eqref{E:str-rel}.  Writing $\bS = s_1 \be_1 + s_2 \be_2 + s_3 \bf_1 + s_4 \bf_2$, we have
 \begin{equation} \label{E:L1L2-br}
 \begin{split} 
 [\ell_1,\ell_2] &\equiv (s_2 - s_1 t - A_2) \be_2 + (s_3 + s_4 t + A_1) \bf_1 \\
 &\qquad + \left(\ell_1(A_2) - \ell_2(A_1) - s_1 A_1 + s_4 A_2\right) \partial_t \quad \mod \sqrt{D}.
 \end{split}
 \end{equation}
 Using part (a) of Proposition \ref{P:XXX-decomp}, we force $[\ell_1,\ell_2] \equiv 0 \mod D$ and obtain the relations \eqref{E:SA1A2}.  This proves the first claim.
  To confirm part (c) of Proposition \ref{P:XXX-decomp}, we now compute:
 \begin{itemize}
 \item $V \otimes \sqrt{D} \cong D^{-2} / D$: Observe $[\partial_t, \ell_1] \equiv \be_2, \, [\partial_t, \ell_2] \equiv \bf_1 \mod D$.
 \item $\sqrt{D} \otimes D^{-2} / D \cong T\widetilde{N} / D^{-2}$: $\begin{pmatrix} [\ell_1, \be_2] & [\ell_1,\bf_1] \\ [\ell_2,\be_2] & [\ell_2,\bf_1] \end{pmatrix} \equiv \begin{pmatrix} 0 & [\be_1,\bf_1]\\ [\be_2,\bf_2] & 0 \end{pmatrix} \mod \widetilde{C}$.
 \end{itemize}
 Composition yields a symmetric bilinear map $\sqrt{D} \otimes \sqrt{D} \to V^* \otimes T\widetilde{N} / D^{-2}$ for which $L_i := \langle \ell_i \rangle$ are null.
 
 For the second claim use \eqref{E:L1L2-br}.  Note that $-s_1 A_1 + s_4 A_2 = \be^1(\bS) \bf^1(\bS) + \be^2(\bS) \bf^2(\bS)$, so we get \eqref{E:SA1A2}.  Since $\bS$ is quadratic in $t$, then $A_i$ are cubic in $t$ and so a priori $\cQ_4$ is quintic in $t$.  However, the order 5 term of $\cQ_4$ agrees with that of $-A_1 \partial_t A_2 + A_2 \partial_t A_1$, which is $t^3 \bf^2([\be_2,\bf_1]) (-3t^2 \be^1([\be_2,\bf_1])) - t^3  \be^1([\be_2,\bf_1]) (-3t^2 \bf^2([\be_2,\bf_1])) = 0$, so $\deg(\cQ_4) \leq 4$.
 \end{proof}

  \begin{remark} \label{R:Q4-transform}
  A local change of LC-adapted framing from $(\be_1,\be_2,\bf_1,\bf_2)$ to $(\widehat\be_1,\widehat\be_2,\widehat\bf_1,\widehat\bf_2)$ is determined by how $(\widehat\be_1,\widehat\be_2)$ differs from $(\be_1,\be_2)$, i.e.\ pointwise, by a $\GL(2)$ transformation.  This induces a fractional linear transformation $\hat{t} = \frac{at+b}{ct+d}$, from which we can verify that $\widehat\cQ_4(\,\hat{t}\,) = \frac{1}{(ct+d)^4} \cQ_4(t)$.
  \end{remark}

 Let us now specialize to an SILC structure.  Locally, this is given by the 2nd order PDE system
 \begin{align}
 w_{11} = \sfF, \quad 
 w_{12} = \sfG, \quad
 w_{22} = \sfH,
 \end{align}
 where $\sfF,\sfG,\sfH$ are functions of $(z^1,z^2,w,w_1,w_2)$.  More precisely, we have a contact 5-manifold $(N,C)$ with $C = E \op F = \langle \be_1, \be_2 \rangle \op \langle \bf_1, \bf_2 \rangle$ given by the LC-adapted framing $\{ \be_1, \be_2, \bf_1, \bf_2 \}$:
 \begin{align} \label{E:SILC-PDE}
 \begin{split}
 &\be_1 = \partial_{z^1} + w_1 \partial_w + \sfF \partial_{w_1} + \sfG \partial_{w_2}, \quad
 \bf_1 = \partial_{w_1}, \\
 &\be_2 = \partial_{z^2} + w_2 \partial_w + \sfG \partial_{w_1} + \sfH \partial_{w_2}, \quad
 \bf_2 = \partial_{w_2}.
 \end{split}
 \end{align}
 
 \begin{cor} \label{C:SILC-5} For the SILC $(N^5;E,F)$ given by \eqref{E:SILC-PDE}, we have
 \begin{align} \label{E:Q4}
 \cQ_4 = \sfF_{qq} + 2t(\sfG_{qq} - \sfF_{pq}) + t^2(\sfF_{pp} - 4 \sfG_{pq} + \sfH_{qq}) + 2t^3(\sfG_{pp} - \sfH_{pq}) + t^4 \sfH_{pp},
 \end{align}
 where $(p,q) := (w_1,w_2)$.  In the ILC case, $\cQ_4$ is the complete obstruction to local equivalence with the flat model $w_{ij} = 0$.
 \end{cor}
 
 \begin{proof}
 Using \eqref{E:SILC-PDE}, we calculate $\bS = [\be_1 + t \be_2, t \bf_1 - \bf_2] =: s_3 \bf_1 + s_4 \bf_2$, where
 \begin{align}
 s_3 = \sfF_q + t (\sfG_q - \sfF_p) - t^2 \sfG_p, \quad
 s_4 = \sfG_q + t (\sfH_q - \sfG_p) - t^2 \sfH_p.
 \end{align}
 Hence, $A_1 = -s_3 - s_4 t$ and $A_2 = 0$ by \eqref{E:SA1A2}, and also $\be^1(\bS) = \be^2(\bS) = 0$.
 Then \eqref{E:Phi} yields $ \cQ_4 = \ell_2(A_1) = (\bf_2 - t \bf_1)(s_3 + s_4 t)$, which simplifies to \eqref{E:Q4} above.
 
 Homogenizing $\cQ_4$ and replacing $t \mapsto -t$, we recover the harmonic curvature expression $\cW$ derived in \cite[(3.3)]{ILC}, which is the complete local obstruction to flatness for 5-dimensional ILC structures.
 \end{proof}

  A key advantage of \eqref{E:Phi} (see next section) is that it can be easily evaluated on homogeneous structures in terms of Lie algebra data.  A PDE realization as in Corollary \ref{C:SILC-5} is not needed.

 By Remark \ref{R:Q4-transform}, the {\sl root type}\footnote{We should always view $\cQ_4$ as a {\em quartic}: e.g.\ when the coefficient of $t^4$ vanishes, we regard $\infty$ as being a root.} of $\cQ_4$ is a discrete invariant of an LC structure.  We denote this by $\sfN$ (quadruple root), $\sfD$ (two double roots), $\mathsf{III}$ (triple root), $\mathsf{II}$ (one double root \& two simple roots), $\mathsf{I}$ (four distinct roots), or $\mathsf{O}$ (identically zero).  Locally, only $w_{ij} = 0$ has constant type $\pO{}$ everywhere.

 \subsection{Symmetries and homogeneous examples}
 For an LC structure $(N;E,F)$, an {\sl automorphism [(infinitesimal) symmetry]} is a diffeomorphism [vector field] of $N$ preserving both $E$ and $F$ under pushforward [Lie derivative]. The symmetry dimension for LC structures $(N^{2n+1};E,F)$ is at most $(n+2)^2-1$ and this upper bound is (locally uniquely) realized by $\fsl(n+2)$ on the flat model $w_{ij} = 0$.    Focusing now on the 5-dimensional ILC case, $15$ is the maximal symmetry dimension, and there is a well-known symmetry gap to the next realizable symmetry dimension, which is 8.  Finer (sharp) upper bounds for structures with constant root type for $\cQ_4$ are also known (see \cite[Thm.3.1]{ILC}):
 \begin{align} \label{E:RootTypeSym}
 \begin{array}{cccccccc}
 \mbox{Root type} & \sfO & \sfN & \sfD & \mathsf{III} & \mathsf{II} & \mathsf{I}\\ \hline
 \mbox{Max. sym. dim.} & 15 & 8 & 7 & 6 & 5 & 5
 \end{array}
 \end{align}
   
Let $G$ be a Lie group and $K$ a closed subgroup.  Any $G$-invariant ILC structure on $N = G/K$ is completely encoded by the following algebraic data generalizing Definition \ref{D:ILC-triple}.
 
 \begin{defn} An {\sl ILC quadruple} $(\fg,\fk;\fe,\ff)$ consists of:
 \begin{enumerate}[(i)]
 \item $\fg$ is a Lie algebra and $\fk$ is a Lie subalgebra;
 \item $\fe$ and $\ff$ are Lie subalgebras of $\fg$ with $\fe \cap \ff = \fk$ (in particular, $[\fk,\fe] \subset \fe$ and $[\fk,\ff] \subset \ff$);
 \item $\dim(\fe/\fk) = \dim(\ff/\fk) = \half( \dim(\fg/\fk)-1)$;
 \item $C := \fe/\fk \op \ff / \fk$ is a non-degenerate subspace of $\fg / \fk$, i.e.\ the map $\eta: \bigwedge^2 C \to \fg / C$ given by $x \wedge y \mapsto [x,y] \mod C$ is non-degenerate.\footnote{Although $\fk$ is not usually an ideal in $\fg$ (so there is no well-defined bracket on $\fg/\fk$ coming from $\fg$), the map $\eta$ is well-defined by (i)--(iii).}
 \item ({\sl Effectivity}) The induced action of $\fk$ on $C$ is non-trivial.
 \end{enumerate}
 When $\fk=0$, we simply refer to $(\fg,0;\fe;\ff)$ as an {\sl ILC triple} $(\fg;\fe,\ff)$.  We will use the notation $\dim(\ILCsym(\fg;\fe,\ff))$ to denote the ILC symmetry dimension of the unique left-invariant ILC structure on any Lie group $G$ with Lie algebra $\fg$ determined by the data $(\fg;\fe,\ff)$.
 \end{defn}
 Given an ILC triple $(\fg;\fe,\ff)$ with $\dim(\fg) = 5$, let $G$ be any Lie group with Lie algebra $\fg$.  Using an LC-adapted framing $\{ \be_1, \be_2, \bf_1, \bf_2 \}$  consisting of left-invariant vector fields on $G$, we see that $A_1$ and $A_2$ are polynomials in $t$ with {\em constant} coefficients, and \eqref{E:Phi} becomes:
  \begin{equation} \label{E:Phi-ST}
 \cQ_4 = - A_1 \partial_t A_2 + A_2 \partial_t A_1 - \be^1(\bS) \bf^1(\bS) - \be^2(\bS) \bf^2(\bS),
 \end{equation}
 where
 \begin{align} \label{E:SA1A2-again}
 \bS = [\be_1 + t \be_2, t \bf_1 - \bf_2], \quad
 A_1 = -(\bf^1+t\bf^2)(\bS), \quad A_2 = (\be^2 - t\be^1 )(\bS).
 \end{align}

 We now consider some examples.  Henceforth, $\{ H, X, Y \}$ will denote a standard $\fsl(2)$-triple satisfying the commutator relations 
 \begin{align} \label{E:sl2-triple}
 [H,X] = 2X, \quad [H,Y]=-2Y, \quad [X,Y]=H.
 \end{align}
 (When appropriate, we regard these as $2\times2$ matrices: $H = \begin{psm} 1 & 0\\ 0 & -1\end{psm}, X = \begin{psm} 0 & 1\\ 0 & 0 \end{psm}, Y = \begin{psm} 0 & 0\\ 1 & 0 \end{psm}$.)

  \begin{ex} \label{X:saff} Consider $\fg = \mathfrak{saff}(2,\bbC) := \fsl(2,\bbC) \ltimes \bbC^2$ and basis $\{ H,X,Y, v_1, v_2 \}$.  Aside from the $\fsl(2)$-triple, the only other non-trivial brackets are:
 \begin{equation} \label{E:saff2-rels}
 [H,v_1] = v_1, \quad [H,v_2] = -v_2, \quad [X,v_2] = v_1, \quad [Y,v_1] = v_2.
 \end{equation}
 Define an ILC triple $(\fg;\fe,\ff)$ via
 \begin{align}
 \fe = \langle H+v_1, X \rangle, \quad \ff = \langle H-v_2, Y \rangle,
 \end{align}
 and an LC-adapted framing:
 \begin{align}
 \be_1 = X, \quad \be_2 = H + v_1 + X, \quad
 \bf_1 = 3Y, \quad \bf_2 = H - v_2 - Y.
 \end{align}
 We compute $\bS = \be_1 + (2t+1)\be_2 - t^2 \bf_1 + t(3t+2) \bf_2$,
hence $A_1 = -t^2 - 3t^3$ and $A_2 = 1 + t$, while $\cQ_4 = -4t(t+1)(3t+1)$, which has distinct roots $\{ -1, -\frac{1}{3}, 0, \infty \}$, so is of root type $\pI{}$.  From \eqref{E:RootTypeSym}, we conclude that $\dim(\ILCsym(\fg;\fe,\ff)) = 5$.
 \end{ex}
 
 If the homogeneous structure is not type $\mathsf{II}$ or $\mathsf{I}$, then the symmetry dimension may be higher than expected.  Algebraically, this amounts to exhibiting:
 
 \begin{defn} \label{D:embedding} An {\sl embedding} of an ILC triple $(\fg;\fe,\ff)$ into an ILC quadruple $(\bar\fg,\bar\fk;\bar\fe,\bar\ff)$ is a  Lie algebra monomorphism $\iota\colon \fg\to\bar\fg$, such that 
 \begin{align}
 \iota(\fg) \cap \bar\fk = 0, \quad \iota(\fe) \subset \bar\fe, \quad \iota(\ff) \subset \bar\ff.
 \end{align}
 If $\fg \subset \bar\fg$ is a subalgebra and $\iota$ is the natural inclusion, we say that $(\bar\fg,\bar\fk;\bar\fe,\bar\ff)$ is an {\sl augmentation} of $(\fg;\fe,\ff)$ by $\bar\fk$.  In particular, $\bar\fg = \fg + \bar\fk$, $\bar\fe = \fe + \bar\fk$, and $\bar\ff = \ff + \bar\fk$.
 \end{defn}
 
 Note that for an augmentation, only the additional brackets involving $\bar\fk$ need to be specified (and Jacobi identity for $\bar\fg$ should be verified).
 
 \begin{ex} \label{X:sl2r2} Consider $\fg=\mathfrak{sl}(2,\bbC) \times \fr_2$, where $\fr_2$ is the unique 2-dimensional non-abelian Lie algebra, and basis $\{ H, X, Y, S, T \}$.  Aside from the $\fsl(2)$-triple, the only other non-trivial bracket is $[S,T] = T$.  Let $\alpha \neq 0$, $\beta \neq 0$, $\alpha\neq \beta$, and define an ILC triple $(\fg;\fe,\ff)$ via:
 \begin{align} \label{E:NS31}
 \fe = \langle H + \alpha S + T, X \rangle, \quad
 \ff = \langle H + \beta S + T, Y \rangle.
 \end{align}
 Here is an LC-adapted framing:
 \begin{align}
 \be_1 = \frac{1}{\beta-\alpha} (H + \alpha S + T), \quad
 \be_2 = X, \quad
 \bf_1 = H + \beta S + T, \quad
 \bf_2 = Y.
 \end{align}
 We compute $\bS = -t\beta \be_1 - 2t^2 \be_2 + \frac{t\alpha}{\beta-\alpha} \bf_1 + \frac{2}{\beta-\alpha} \bf_2$, hence $A_1 = \frac{t(\alpha+2)}{\alpha-\beta}$, $A_2 = t^2(\beta-2)$, and
 \begin{align}
 \cQ_4 = \frac{2(\alpha\beta + \beta - \alpha)}{\beta-\alpha} t^2.
 \end{align}
 Thus, the ILC structure is type $\sfO$ (hence, 15-dimensional symmetry) when $\alpha\beta = \alpha - \beta$, and type $\sfD$ otherwise (hence, at most 7-dimensional symmetry by \eqref{E:RootTypeSym}).  In the latter case, we now show that it is indeed 7-dimensional and is a realization of model \textsf{D.7} from \cite{ILC}.
 
 Let $\bar\fg = \fsl(2,\bbC)\times \fsl(2,\bbC) \times \bbC$ with basis $\{ H_1,X_1,Y_1,H_2,X_2,Y_2,Z \}$ consisting of $\fsl(2)$-triples $\{ H_i,X_i,Y_i \}$ and central element $Z$.  Given $\lambda \in \bbC^\times$, define an ILC quadruple $(\bar\fg,\bar\fe;\bar\ff,\bar\fk)$:
\begin{align}
\bar\fk = \langle H_1-Z, \lambda H_2-Z\rangle, \quad
\bar\fe = \langle X_1, X_2\rangle + \bar\fk, \quad
\bar\ff = \langle Y_1, Y_2\rangle + \bar\fk.
\end{align}

 For any $t \in \bbC$, define a monomorphism $\iota\colon\fg\to\bar\fg$ sending $H \mapsto H_1$, $X \mapsto X_1$, $Y \mapsto Y_1$, and
 \begin{align}
 \begin{cases}
 S \mapsto -\tfrac{\alpha+\beta}{2(\alpha-\beta)}H_2 + \tfrac{\beta}{\alpha-\beta}X_2-\tfrac{\alpha}{\alpha-\beta} Y_2 +tZ, \\
 T \mapsto + \tfrac{\alpha\beta}{\alpha-\beta}H_2 - \tfrac{\beta^2}{\alpha-\beta} X_2
 + \tfrac{\alpha^2}{\alpha-\beta}Y_2.
 \end{cases}
 \end{align}
which implies
\begin{align}
   \iota( H + \alpha S + T) &= H_1 -\tfrac{\alpha}{2} H_2 + \beta X_2 + \alpha tZ,\\
   \iota( H + \beta S + T ) &= H_1 + \tfrac{\beta}{2} H_2 + \alpha Y_2 + \beta tZ.
\end{align}
 Thus, $\iota(\fe)\subset \bar\fe$ and $\iota(\ff)\subset\bar\ff$ if and only if $\lambda (\alpha t+1) =\tfrac{\alpha}2$ and $\lambda (\beta t+1) = -\tfrac{\beta}2$.
Solving yields $t=-\frac{\alpha+\beta}{2\alpha\beta}$ and $\lambda=\frac{\alpha\beta}{\beta-\alpha} \in \bbC \backslash \{ 0,-1 \}$. (Recall $\alpha\beta\neq \alpha-\beta$ for non-flatness.)  These parameters uniquely define $\iota$ and provide an embedding from $(\fg;\fe,\ff)$ into $(\bar\fg,\bar\fk;\bar\fe,\bar\ff)$ for $\lambda=\frac{\alpha\beta}{\beta-\alpha}$.
Thus, $\dim(\ILCsym(\fg;\fe,\ff))$ is 15 when $\alpha\beta = \alpha-\beta$ and 7 otherwise.
 \end{ex}
 \section{Cases without 3-dimensional abelian ideals}
 \label{S:NoAb}
 Given an ILC triple $(\fg;\fe,\ff)$ an {\sl admissible anti-involution} is an anti-automorphism $\tau : \fg \to \fg$ with $\tau^2 = \id$ that swaps $\fe$ and $\ff$.  In this section, we will prove the following result:
 
 \begin{theorem} \label{T:NoAb} Let $\fg$ be a 5-dimensional complex Lie algebra without 3-dimensional abelian ideals.  There is a unique (up to isomorphism) ASD-ILC triple $(\fg;\fe,\ff)$ with $\dim(\ILCsym(\fg;\fe,\ff)) = 5$.  Namely, $\fg \cong \saff(2,\bbC)$ together with $\fe$ and $\ff$ given by \eqref{E:NS2-EF}, and such $(\fg;\fe,\ff)$ has a unique admissible anti-involution.
 \end{theorem}
 
 The proof begins by establishing (in Proposition \ref{P:NoAb}) the classification of all 5-dimensional {\em complex} $\fg$ without 3-dimensional abelian ideals.  For each $\fg$ in this list, we investigate the ASD-ILC triples $(\fg;\fe,\ff)$ that it can support, but discard those with $\dim(\ILCsym(\fg;\fe,\ff)) \geq 6$.
 
 \subsection{A key classification result} 
 \label{SS:NoAb}
 A feature of the proof of the following result is its independence of the known Mubarakzyanov classification of 5-dimensional {\em real} Lie algebras \cite{mub1963}. 
 
 \begin{prop} \label{P:NoAb} Any 5-dimensional complex Lie algebra $\fg$ without 3-dimensional abelian ideals is isomorphic to one of the following:
 \begin{enumerate}
 \item[(NS1)] $\fsl(2,\bbC) \times \bbC^2$;
 \item[(NS2)] $\fsl(2,\bbC) \ltimes \bbC^2$;
 \item[(NS3)] $\fsl(2,\bbC) \times \fr_2$, where $\fr_2$ is a 2-dimensional non-abelian Lie algebra;
 \item[(SOL)] the Lie algebra of upper-triangular matrices in $\fsl(3,\bbC)$.
 \end{enumerate}
 \end{prop}
 
 \begin{proof}
Consider the following cases.
 \begin{enumerate}
 \item \emph{$\fg$ is non-solvable.} By the Levi decomposition, $\fg \cong \fsl(2,\bbC) \ltimes \rad(\fg)$, where $\dim(\rad(\fg)) = 2$.  If $\rad(\fg)$ is abelian, then we get either (NS1) or (NS2).  Otherwise,  $\rad(\fg) \cong \fr_2$ and $\fsl(2,\bbC)$ acts trivially on it (since $\Der(\fr_2)$ is solvable) and we get (NS3).

 \item \emph{$\fg$ is solvable, but not nilpotent.} Let $\fn$ be the nilradical (i.e.\ maximal nilpotent ideal) of $\fg$, which coincides with the set of all nilpotent elements in $\fg$.  If $\fg$ has center $\cZ(\fg)$, then
 \begin{align}
 4 \geq \dim\,\fn \ge \half(\dim\,\fg + \dim\,\cZ(\fg)),
 \end{align}
so $\dim\,\fn = 3$ or $4$.  (See \cite{mub1966}, \cite[Thm.5.2]{OV1994} for the second inequality.) Consider $\rho\colon \fg \mapsto \Der(\fn),\quad u\mapsto \ad\,u|_{\fn}.$

 \begin{enumerate}
 \item \underline{$\dim(\fn)=3$}: by assumption, $\fn$ is non-abelian, so $\fn \cong \fn_3$, the 3-dimensional Heisenberg Lie algebra.  In a basis $\{P,Q,R\}$ of $\fn$ with only non-trivial bracket $[P,Q]=R$, we have:
\[
\Der(\fn_3)=\begin{pmatrix} 
	a_{11} & a_{12} & 0 \\
	a_{21} & a_{22} & 0 \\
	b_{1} & b_{2} & a_{11}+a_{22}
	\end{pmatrix}, \quad 
	\rho(\fn_3)=\begin{pmatrix} 
	0 & 0 & 0 \\
	0 & 0 & 0 \\
	b_{1} & b_{2} & 0
	\end{pmatrix}.
\]
In particular, $\Der(\fn_3)/\rho(\fn_3)\cong\gl(2,\bbC)$. 
By maximality of $\fn$, $\rho(T)$ is not nilpotent for any $T \not\in \fn$.  Let $\{S_1, S_2\}$ be a basis of a complementary subspace to $\fn$. Then $[S_1,S_2]\subset [\fg,\fg]\subset \fn$, and hence $\{\rho(S_1),\rho(S_2)\}\mod \rho(\fn_3)$ would form a basis of a commutative subalgebra in $\Der(\fn_3)/\rho(\fn_3)\cong\gl(2,\bbC)$ consisting of non-nilpotent elements (except for zero). But the only such subalgebra is conjugate to the subalgebra of diagonal matrices in $\gl(2,\bbC)$. So, adjusting elements $S_1$ and $S_2$ by $\fn_3$ if needed, we can assume that $\rho(S_1)=\diag(1,0,1)$ and $\rho(S_2)=\diag(0,1,1)$.

Let $[S_1,S_2]=u\in\fn_3$. Since $\rho(u)=\rho([S_1,S_2])=0$, we get that $u\in \cZ(\fn_3)$ and, thus, $u=\alpha R$ for some $\alpha\in\bbC$. Replacing $S_1$ by $S_1+\alpha R$ we can normalize $\alpha$ to $0$.  

Thus, $\fg$ is isomorphic to (SOL) via the map:
\begin{equation}\label{E:SOL-basis}
\begin{gathered}
P\mapsto \left(\begin{smallmatrix} 0 & 1 & 0 \\ 0 & 0 & 0 \\ 0 & 0 & 0 \end{smallmatrix}\right),\ 
Q\mapsto \left(\begin{smallmatrix} 0 & 0 & 0 \\ 0 & 0 & 1 \\ 0 & 0 & 0 \end{smallmatrix}\right),\ 
R\mapsto \left(\begin{smallmatrix} 0 & 0 & 1 \\ 0 & 0 & 0 \\ 0 & 0 & 0 \end{smallmatrix}\right),\\
S_1\mapsto \left(\begin{smallmatrix} \frac{2}{3} & 0 & 0 \\ 0 & -\frac{1}{3} & 0 \\ 0 & 0 & -\frac{1}{3} \end{smallmatrix}\right),\ 
S_2\mapsto \left(\begin{smallmatrix} \frac{1}{3} & 0 & 0 \\ 0 & \frac{1}{3} & 0 \\ 0 & 0 & -\frac{2}{3} \end{smallmatrix}\right).
\end{gathered}
\end{equation} 

\item \underline{$\dim(\fn) = 4$}: Let $S\in\fg$ be any non-zero element not contained in $\fn$.  The Lie algebra $\fn$ is isomorphic to one of the three possible nilpotent algebras in dimension~$4$:
\begin{enumerate}
	\item $\fn=\bbC^4$. Then $\rho(S)$ necessarily preserves a 3-dimensional subspace in $\fn$, which will be an abelian ideal in $\fg$.
	\item $\fn=\fn_3\times \bbC$. It has a 2-dimensional center $\cZ(\fn)$. The action of $\rho(S)$ on $\fn/\cZ(\fn)$  preserves a one-dimensional subspace, whose pre-image in $\fn$ is an abelian ideal.
	\item $\fn=\fn_4$ with a basis $\{P,Q_1,Q_2,Q_3\}$ and non-zero brackets $[P,Q_1]=Q_2$, $[P,Q_2]=Q_3$.  Then the second element $\cZ_2(\fn)$ in the upper central series of $\fn$ is equal to $\langle Q_2,Q_3\rangle$.  Its centralizer is equal to $\langle Q_1, Q_2,Q_3\rangle$ and is an abelian ideal in $\fg$.  
\end{enumerate} 
\end{enumerate}
\item \emph{$\fg$ is nilpotent.} Let $\fa$ be a maximal abelian ideal of $\fg$. As in the previous case, consider the representation:
\[
   \rho\colon \fg \to \gl(\fa),\quad u\mapsto \ad\,u|_{\fa}.
\]
Let us show that $\ker \rho = \fa$. Indeed, otherwise the centralizer $\cZ_{\fg}(\fa)$ of $\fa$ in $\fg$ is strictly greater than $\fa$. Since $\fg$ is nilpotent, by Engel's theorem we can construct a sequence of ideals of $\fg$:
\[
  \fa \subset \fa_1 \subset \dots \subset \fa_r=\cZ_{\fg}(\fa)
\]
such that $\dim\,\fa_i = \dim\,\fa+i$ for $i=1,\dots,r$. But then $\fa_1$ is also abelian, which contradicts the maximality of $\fa$. 

So, if $\dim\,\fa=n$, then $\rho(\fg)$ is a subalgebra in $\gl(\fa)$ consisting of nilpotent elements. Then by Engel's theorem we get $\dim\, \fg/\fa\le n(n-1)/2$ and $\dim\,\fg\le n(n+1)/2$. Thus, we see that $n\ge 3$. 

The cases $n=3$ and $n=5$ are ruled out by hypothesis. Finally, if $n=4$, then, as in the solvable case with $\fn=\bbC^4$, we can find a 3-dimensional ideal in $\fa$.  
\end{enumerate}
\end{proof}
 \subsection{NS1} 
 For $\fg=\fsl(2,\bbC) \times \bbC^2$, if $(\fg;\fe,\ff)$ is an ILC triple, then the 2-dimensional center $\cZ(\fg) = \bbC^2$ must have non-trivial intersection with $C=\fe\oplus\ff$.  But this contradicts the non-degeneracy of $C$, so no such ILC triples exist.
 \subsection{NS2} 
 For $\fg=\mathfrak{saff}(2,\bbC) = \mathfrak{sl}(2,\bbC) \ltimes \bbC^2$, we use notation introduced in Example \ref{X:saff}. 
  
\begin{prop} \label{P:saff} For $\fg=\mathfrak{saff}(2,\bbC)$, 
any ASD-ILC triple $(\fg;\fe,\ff)$ is $\Aut(\fg)$-equivalent to:
\begin{align} \label{E:NS2-EF}
\fe = \left\langle H + v_1,  X\right\rangle, \quad \ff =\left\langle H - v_2, Y \right\rangle.
\end{align}
\end{prop}

\begin{proof} Observe that $\bbC^2 = \rad(\fg)$, so it is preserved by any anti-involution.  Assuming $\fe \cap \bbC^2 \neq 0$, then $\ff \cap \bbC^2 \neq 0$ has the same dimension by the ASD property.  In this case, $\fe \cap \ff = 0$ implies $\bbC^2 \subset C = \fe \oplus \ff$.  But $\bbC^2 \subset \fg$ is an ideal, so this contradicts non-degeneracy of $C$.  Thus, we can assume that $\fe \cap \bbC^2 = \ff \cap \bbC^2 = 0$.
	
Consider the quotient homomorphism $\pi\colon \fg\to \fg/\bbC^2=\fsl(2,\bbC)$.  Since $\fe$ and $\ff$ are both transverse to $\bbC^2$, then $\pi(\fe)$ and $\pi(\ff)$ are both 2-dimensional subalgebras of $\fsl(2,\bbC)$ that are {\em distinct}.  (If $\pi(\fe) = \pi(\ff)$, then $C = \fe \oplus \ff \subset \ff \ltimes \bbC^2$, hence $C = \ff \ltimes \bbC^2$ since both have dimension 4.  But $\ff \oplus \bbC^2$ is a subalgebra, which contradicts non-degeneracy of $C$.)

 Any 2-dimensional subalgebra of $\fsl(2,\bbC)$ coincides with the isotropy of some line in $\bbC^2$.  Since $\SL(2,\bbC)$ acts transitively on pairs of distinct lines in $\bbC^2$, then we can assume up to $\Aut(\fg)$ that $\pi(\fe) \equiv \langle H, X \rangle$ and $\pi(\ff) \equiv \langle H,Y \rangle$.  Closure under the Lie bracket implies:
 \begin{align}
\fe = \left\langle H  + \begin{psm} a_1 \\ b_1 \end{psm},  X - \begin{psm} b_1 \\ 0 \end{psm} \right\rangle, \quad
\ff =\left\langle H  + \begin{psm} a_2 \\ b_2 \end{psm},  Y  + \begin{psm} 0 \\ a_2 \end{psm} \right\rangle,
 \end{align}
where we identify $v_1 = \begin{psm} 1\\0\end{psm}$ and $v_2 = \begin{psm} 0\\1\end{psm}$.  Note that $\Aut(\fg)$ contains the following:
\begin{enumerate}
\item[(i)] translations of $\bbC^2$ induce $(a_1,b_1,a_2,b_2) \mapsto (a_1 + r, b_1 + s, a_2 + r, b_2 + s)$ for any $r,s \in \bbC$.  We use this to normalize $a_2 = b_1 = 0$.
\item[(ii)] the scaling $(v_1,v_2,H,X,Y) \mapsto (\lambda v_1, \mu v_2, H, \frac{\lambda}{\mu} X, \frac{\mu}{\lambda} Y)$ for any $\lambda,\mu \in \bbC^\times$.  This induces the scaling $(a_1,b_2) \mapsto (\lambda a_1, \mu b_2)$.
\item[(iii)] the swap $(v_1,v_2,H,X,Y) \mapsto (v_2,v_1,-H,Y,X)$ induces $(a_1,b_2) \mapsto (-b_2,-a_1)$. 
\end{enumerate}
 Since $\fe \cap \ff =0$, then $(a_1,b_2) \neq (0,0)$.  Using (iii), we may assume that $a_1 \neq 0$, and then normalize $a_1 = 1$ using (ii).  
 \begin{itemize}
 \item $b_2 \neq 0$: Using (ii), normalize to $b_2 = -1$.  Then (iii) determines both a residual involution as well as an anti-involution.
 \item $b_2 = 0$: $\fe = \left\langle H  + v_1,  X \right\rangle$ and $\ff =\left\langle H,  Y \right\rangle$. But clearly $[X,\,\cdot\,] \equiv 0\mod C$, which contradicts non-degeneracy of $C$.
 \end{itemize}
\end{proof}

 From Example \ref{X:saff}, we saw that \eqref{E:NS2-EF} has root type $\mathsf{I}$ and $\dim(\ILCsym(\fg;\fe,\ff)) = 5$.
 
 \begin{prop} \label{P:saff-AI} For $(\fg;\fe,\ff)$ as in Proposition \ref{P:saff}, the unique admissible anti-involution $\tau$ is:
 \begin{align} \label{E:saff-AI}
 (H,X,Y,v_1,v_2) \mapsto (-H,Y,X,v_2,v_1).
 \end{align}
 \end{prop}
 
 \begin{proof} Since $\fe$ and $\ff$ are non-abelian, then $\tau$ must swap the lines $[\fe,\fe] = \langle X \rangle$ and $[\ff,\ff] = \langle Y \rangle$.  These act on the radical $\rad(\fg) = \bbC^2 = \langle v_1, v_2 \rangle$ with images $\langle v_1 \rangle$ and $\langle v_2 \rangle$ respectively.  Since $0 \neq \tau(v_1) = \tau([X,v_2]) = [\tau(X),\tau(v_2)]$ and $\tau(X) \in \langle Y \rangle$, we deduce that $\tau$ must swap $\langle v_1 \rangle$ and $\langle v_2 \rangle$.  Finally, $\tau$ must preserve $\langle H \rangle$, which is the intersection of the normalizers of the above four lines $\langle X \rangle$, $\langle Y \rangle$, $\langle v_1 \rangle$, $\langle v_2 \rangle$.  Since $\tau$ is admissible, it preserves $\fe$ and $\ff$, so $(H,X,Y,v_1,v_2) \stackrel{\tau}{\mapsto} (a H, b Y, c X, -a v_2, -a v_1)$.  Using \eqref{E:sl2-triple} and \eqref{E:saff2-rels}, the anti-involution property forces $(a,b,c) = (-1,1,1)$. \qedhere
 \end{proof}
 \subsection{NS3}
 Let $\fg=\mathfrak{sl}(2,\bbC) \times \fr_2$.  The $\fsl(2,\bbC)$ factor is the second derived algebra of $\fg$, while $\fr_2 = \rad(\fg)$, so both are preserved under any anti-involution.  Fix a basis $\{ H, X, Y, S, T \}$ as in Example \ref{X:sl2r2}.  Observe that $\Aut(\fr_2)$ consists of the transformations
 \begin{align} \label{E:Aut-r2}
 (S,T)\mapsto (S+rT, \lambda T), \quad r\in \bbC, \quad \lambda\in\bbC^\times.
 \end{align}

\begin{prop}
	Let $\fg=\mathfrak{sl}(2,\bbC) \times \fr_2$.  Any ASD-ILC triple $(\fg;\fe,\ff)$ has $\dim(\ILCsym(\fg;\fe,\ff)) \geq 6$.
\end{prop}

\begin{proof} Let $\pi_1 \colon \fg\to \fsl(2,\bbC)$ and $\pi_2 \colon \fg \to \fr_2$ be the natural projections.  As in the previous case, we may assume that $\pi_1(\fe) = \langle H,X \rangle$ and $\pi_1(\ff) = \langle H,Y \rangle$.  Thus,
 \begin{align}
 \fe = \left\langle H + a_1S+b_1T,  X+c_1S+d_1T \right\rangle,
 \end{align}
 which is a subalgebra if and only if $c_1 = 0$ and $(a_1-2)d_1=0$.
 \begin{enumerate}[(i)]
 \item \underline{$d_1=0$}: We have $[\fe,\fe]\subset \fsl(2,\bbC)$.  By the ASD property, $\ff$ satisfies $[\ff,\ff] \subset \fsl(2,\bbC)$.  Then
 \begin{align}
 \fe = \left\langle H + a_1S+b_1T,  X \right\rangle, \quad \ff = \left\langle H + a_2S+ b_2T,  Y \right\rangle.
 \end{align}
 Assume that $a_1=0$.  Then $\pi_2(\fe) \subset [\fr_2,\fr_2] = \langle T \rangle$.  Stability under any anti-involution implies that $a_2=0$.  But then $C=\fe\oplus\ff$ contains $[\fr_2,\fr_2]=\langle T\rangle$, which is an ideal in $\fg$. This contradicts non-degeneracy of $C$.  Thus, $a_1 \neq 0$ and similarly $a_2 \neq 0$. Note that $a_1\ne a_2$ as otherwise we again would have $\langle T \rangle \subset C$. 
	 
 The transformations \eqref{E:Aut-r2} induce $(a_1,b_1,a_2,b_2) \mapsto (a_1,b_1 \lambda + a_1 r,a_2,b_2 \lambda + a_2 r)$, which we use to normalize $b_1=b_2$.  If $b_1 = b_2 = 0$, then $C=\fe\oplus\ff=\fsl(2,\bbC)+\langle S\rangle$, which is degenerate (moreover, a subalgebra in $\fg$).  So, we can assume that $b_1=b_2\ne 0$ and rescale them to $1$. This gives us \eqref{E:NS31} with $\alpha\beta \neq 0, \alpha \neq \beta$.  In Example \ref{X:sl2r2}, we saw these are either type $\sfD$ or $\sfO$, with 7 or 15 symmetries respectively.\\

 \item \underline{$d_1\ne 0$}: Then $a_1=2$ and arguing similarly we obtain
 \begin{align}
 \fe = \left\langle H + 2S+b_1T,  X+d_1T \right\rangle, \quad \ff = \left\langle H -2S+b_2T, Y +d_2 T\right\rangle,
 \end{align}
	 where $d_2\ne 0$. Now conjugation by $\diag(\mu,\frac{1}{\mu}) \in \SL(2,\bbC)$ induces $(d_1,d_2)\mapsto (\frac{d_1}{\mu^2}, d_2 \mu^2)$, which, together with $\Aut(\fr_2)$, allows us to normalize $d_1=d_2=1$.  Using the remaining transformations $S \mapsto S+rT$ in $\Aut(\fr_2)$, we normalize $b_1=b_2$ and obtain:
 \begin{align}
 \fe = \left\langle H + 2S+\alpha T,  X+T \right\rangle, \quad
 \ff = \left\langle H  - 2S+\alpha T,  Y +T\right\rangle \quad
 (\alpha^2+4\ne 0).
 \end{align}
 The condition $\alpha^2+4 \neq 0$ is equivalent to $C=\fe \op \ff$ being non-degenerate.
 
 We now exhibit an embedding of $(\fg;\fe,\ff)$ into some $(\bar\fg,\bar\fk;\bar\fe,\bar\ff)$.  Consider $\bar\fg=\fsl(2,\bbC) \times \fsl(2,\bbC)$ with basis $\{ H_1, X_1, Y_1, H_2, X_2, Y_2 \}$ consisting of two $\fsl(2)$-triples.  Given $\alpha \neq 0$, define $\lambda=-\tfrac{\alpha}{2\sqrt{\alpha^2+4}} \in \bbC \backslash \{ 0, \pm \half \}$ and an ILC quadruple $(\bar\fg,\bar\fk;\bar\fe,\bar\ff)$ \cite[Model $\sfD$.6-3]{ILC} by:
 \begin{align}
 \begin{split}
 \bar\fk &= \langle H_1 - H_2 \rangle, \\
 \bar\fe &= \langle X_1 + \tfrac{2\lambda-1}{2\lambda+1} Y_2, X_2 + \tfrac{2\lambda-1}{2\lambda+1} Y_1 \rangle + \bar\fk, \\ 
 \bar\ff &= \langle X_1 + Y_2, X_2 + Y_1 \rangle + \bar\fk.
 \end{split}
 \end{align}
 We confirm that the following is a monomorphism $\iota\colon\fg\to\bar\fg$ with $\iota(\fe) \subset \bar\fe$ and $\iota(\ff) \subset \bar\ff$:
 \begin{align}
 \begin{split}
H &\mapsto \tfrac{\alpha}{\sqrt{\alpha^2+4}}(-\tfrac{2\lambda+1}{2\lambda}X_1-H_1+\tfrac{2\lambda-1}{2\lambda}Y_1), \\
X &\mapsto \tfrac{1}{\sqrt{\alpha^2+4}}(-\tfrac{2\lambda+1}{2\lambda-1}X_1 - H_1+\tfrac{2\lambda-1}{2\lambda+1}Y_1),\\
Y &\mapsto \tfrac{1}{\sqrt{\alpha^2+4}}(-X_1-H_1+Y_1), \\
S &\mapsto -\tfrac12(X_2+Y_2), \\
T &\mapsto \tfrac{1}{\sqrt{\alpha^2+4}}(X_2+H_2-Y_2).
 \end{split}
 \end{align}
 Finally, when $\alpha=0$, we use the LC-adapted framing
 \begin{align}
 \be_1 = X+T, \quad \be_2 = H + 2S, \quad \bf_1 = H  - 2S, \quad \bf_2 = Y +T
 \end{align}
 to compute $\bS = [\be_1 + t\be_2,t\bf_1 - \bf_2] = -2t \be_1 - \half \be_2 - \half \bf_1 + 2t \bf_2$ and confirm that $\cQ_4 = 0$.
 \end{enumerate}
  \end{proof}

 \subsection{SOL}
Let $\fg = \fb$ be the Lie algebra of upper-triangular matrices in $\fsl(3,\bbC)$.  Consider the basis $\{ S_1, S_2, P, Q, R \}$ from \eqref{E:SOL-basis}, which has non-trivial brackets
 \begin{align}
[S_1,P]=P, \quad [S_1,R]=R, \quad [S_2,Q]=Q, \quad [S_2,R]=R, \quad [P,Q]=R.
 \end{align}
 This has nilradical $\fn_3 = \langle P,Q,R \rangle$, which  agrees with the first derived algebra of $\fg$, so is preserved under any anti-involution.
 \begin{prop} \label{P:S-triples}
 Let $\fg = \fb \subset \fsl(3,\bbC)$.  Any ASD-ILC triple $(\fg;\fe,\ff)$ has $\dim(\ILCsym(\fg;\fe,\ff)) = 15$.
 \end{prop}
 
\begin{proof} Consider two cases:
\begin{enumerate}[(i)]
\item \underline{$\fe\cap \fn_3 = 0$}: 
Let us normalize $\fe = \langle S_1 + \alpha_1 P + \beta_1 Q + \gamma_1 R, S_2 + \alpha_2 P + \beta_2 Q + \gamma_2 R \rangle$ using $\exp(\ad\, \fn_3)$.  Using $\exp(\ad_{t_1 P + t_3 R})$ and then $\exp(t_2 \,\ad_Q)$, we normalize $\alpha_1 = \gamma_1 = \beta_2 = 0$.  Since $\fe$ is a subalgebra, then $\alpha_2 = \beta_1 = \gamma_2 = 0$, so $\fe = \langle S_1, S_2 \rangle$.  Since $\fe$ is abelian and $\fe \cap \fn_3 = 0$, then (by ASD) $\ff$ is abelian and $\ff\cap\fn_3=0$, which yield 
 \begin{align} \label{E:S.1} \tag{S.1}
 \fe = \langle S_1, S_2 \rangle, \quad
 \ff = \langle S_1 + a_1 P + c_1 R, S_2 + b_2 Q + c_2 R \rangle,
 \end{align}
 where $c_2 := c_1 - a_1 b_2$.  Non-degeneracy of $C=\fe\op\ff$ is equivalent to $c_1 c_2 \neq 0$.\\

\item \underline{$\fe \cap \fn_3 \neq 0$}: Assuming $\fe \subset \fn_3$, then $\ff \subset \fn_3$ (by ASD), hence $C=\fe\op\ff \subset \fn_3$, which is a contradiction, so $\dim(\fe \cap \fn_3) = \dim(\ff\cap \fn_3) = 1$.  Also, $\fe\cap\fn_3 \ne \langle R \rangle$ and $\ff \cap \fn_3 \neq \langle R \rangle$, otherwise $\fe$ or $\ff$ would contain an ideal of $\fg$,  contradicting non-degeneracy of $C$.  Note $(S_1,S_2,P,Q,R) \mapsto (S_2,S_1,Q,P,-R)$ is an automorphism, so swapping $P,Q$ if necessary, we may assume that $\fe \cap \fn_3 = \langle P+a_0 Q + a_1 R \rangle$.  For the normalizer $\cN(\fe \cap \fn_3)$:
\begin{align} \label{E:eN}
 \fe \subset \cN(\fe \cap \fn_3) = \begin{cases}
 \langle S_1 + S_2, P+ a_0 Q, R\rangle, & a_0 \neq 0;\\
 \langle S_1,S_2,P,R \rangle, & a_0 = 0.
 \end{cases}
\end{align}
Assume $a_0 \neq 0$.  Then $\dim(\cN(\fe\cap \fn_3)) = 3 = \dim(\cN(\ff\cap \fn_3))$ by ASD, and $C \subset \langle S_1+S_2\rangle \ltimes \fn_3$, so $C$ would be degenerate.  Thus, $a_0 = 0$.

 Note that if $\ff \cap \fn_3 = \langle P + b_0 Q + b_1 R \rangle$, then $b_0 = 0$ as above, while \eqref{E:eN} implies that $C \subset \langle S_1,S_2,P,R \rangle$, so $C$ would be degenerate.  Thus, $\fe\cap \fn_3 =\langle P + \alpha R \rangle$ and $\ff\cap \fn_3 =\langle Q + \beta R \rangle$.
 Using $\exp(\ad\, \fn_3)$, we normalize $\alpha=\beta=0$. Then:
 \begin{align}
 \fe = \langle \alpha_{11} S_1 + \alpha_{12} S_2 + \gamma_1 R, P \rangle, \quad
 \ff = \langle \alpha_{21} S_1 + \alpha_{22} S_2 + \gamma_2 R, Q \rangle
 \end{align}
 \begin{enumerate}[(a)]
 \item \underline{$\fe$ \& $\ff$ non-abelian}: We may assume $\alpha_{11} = \alpha_{22} = 1$.  Use $\exp(t\,\ad_R)$ to normalize $\gamma_1 = 0$.  Since $\gamma_2 \neq 0$ by non-degeneracy, we may normalize $\gamma_2 = 1$.  Then:
 \begin{align} \label{E:S.2} \tag{S.2}
 \fe = \langle S_1 + \alpha S_2, P \rangle, \quad
 \ff = \langle \beta S_1 + S_2 + R, Q \rangle.
 \end{align}

 \item \underline{$\fe$ \& $\ff$ abelian}: $\alpha_{11} = \alpha_{22} = 0$.  Note $\alpha_{12}\alpha_{21} \neq 0$, otherwise $\fn_3 \subset C$, and so $C$ would be degenerate.  Using $\exp(t\,\ad_R)$, we normalize $\gamma_2 = 0$, so we may assume:
 \begin{align} \label{E:S.3} \tag{S.3}
 \fe = \langle S_2 + \gamma R, P \rangle, \quad
 \ff = \langle S_1, Q \rangle.
 \end{align}
 \end{enumerate}
 \end{enumerate}	
We confirm $\cQ_4 = 0$ in all three cases using LC-adapted framings and \eqref{E:Phi-ST}:
 \[
 \begin{array}{cccccc}
 & \be_1 & \be_2 & \bf_1 & \bf_2 & \bS\\ \hline
 \eqref{E:S.1} & S_2 & S_1 & \frac{1}{c_1} (S_1+a_1 P) + R & \frac{1}{c_2} (S_2+b_2 Q)+ R & \frac{1}{c_2} \be_1 - \frac{t^2}{c_1} \be_2 + t^2 \bf_1 - \bf_2\\
 \eqref{E:S.2} & S_1+\alpha S_2 & (1+\alpha) P & \beta S_1+S_2+R & Q & -t^2\beta \be_2 - \alpha \bf_2\\
 \eqref{E:S.3} & S_2+\gamma R & P & -\frac{1}{\gamma} S_1 & Q & \frac{t^2}{\gamma} \be_2 - \bf_2
 \end{array}
 \]
\end{proof}

These ILC structures are all flat.  The proof of Theorem \ref{T:NoAb} is now complete.

 \section{Cases with a 3-dimensional abelian ideal}
 \label{S:NTAb}
 In this section, we prove the following, which will reduce (see \S\ref{S:tubular}) the remainder of our study to tubes on an affinely homogeneous base (Corollary \ref{C:tubes}).
 
  \begin{theorem} \label{T:NTab}  Let $\fg$ be a 5-dimensional complex Lie algebra with a 3-dimensional abelian ideal $\fa$, and  $(\fg;\fe,\ff)$ an ASD-ILC triple with an admissible anti-involution $\tau$.  Suppose that we have $\dim(\ILCsym(\fg;\fe,\ff)) = 5$.  Then $\fa = \tau(\fa)$ with $\fe \cap \fa = \ff \cap \fa = 0$.
 \end{theorem}
  
 We split the proof according to $\fa \neq \tau(\fa)$ or $\fa = \tau(\fa)$.  Finally, we show that $\fa$ is self-centralizing.

 \subsection{The $\fa \neq \tau(\fa)$ case}
 \begin{prop}
 Let $\fg$ be a 5-dimensional complex Lie algebra with a 3-dimensional abelian ideal $\fa$, and  $(\fg;\fe,\ff)$ an ASD-ILC triple with an admissible anti-involution $\tau$.  Suppose that $\fa \neq \tau(\fa)$.  Then:
 \begin{enumerate}[(a)]
 \item $\dim(\fa \cap \tau(\fa)) = 1:$ we have $\dim(\ILCsym(\fg;\fe,\ff)) = 15$;
 \item $\dim(\fa \cap \tau(\fa)) = 2:$ we have $\dim(\ILCsym(\fg;\fe,\ff)) \geq 6$.
 \end{enumerate}
 \end{prop}
 \begin{proof} Since $\fa$ and $\tau(\fa)$ are ideals in $\fg$, then so are $\fn:=\fa + \tau(\fa)$ and $\fa \cap \tau(\fa)$.  Note that 
 \begin{align} \label{E:nn}
 [\fn,\fn] = [\fa,\tau(\fa)] \subset \fa \cap \tau(\fa)  \subset \cZ(\fn).
 \end{align}
 \begin{enumerate}[(a)]
 \item We have $\dim(\fn) = 5$, so $\fn=\fg$.  Since $\dim(\fa \cap \tau(\fa)) = 1$ and $C=\fe\oplus\ff$ is non-degenerate, then \eqref{E:nn} implies $0 \neq [\fg,\fg] = \fa \cap \tau(\fa) = \langle T \rangle$ is transverse to $C$.  Since $\fe$ is a subalgebra, then $[\fe,\fe] \subset \fe \cap [\fg,\fg] = \fe \cap \langle T \rangle = 0$, so $\fe$ is abelian and similarly for $\ff$.  Letting $\{ \be_1, \be_2, \bf_1, \bf_2 \}$ be an LC-adapted framing, the only non-trivial brackets (after rescaling $T$ if necessary) are
 \begin{align}
 [\be_1, \bf_1] = T, \quad [\be_2, \bf_2] = T.
 \end{align}
 Thus, $\fg$ is isomorphic to the 5-dimensional Heisenberg Lie algebra.  By \eqref{E:Phi-ST}, we find that $\cQ_4 = 0$, so we have the flat ILC structure with 15-dimensional symmetry.
 \item Given $N_1 \in \fa$ with $N_1 \not\in \tau(\fa)$, define $N_2 := \tau(N_1) \in \tau(\fa)$, so $N_2 \not\in \fa$.  Since $\dim(C) = \dim(\fn) = 4$, then $\dim(C \cap \fn) \geq 3$, so $\fn$ must be non-abelian (by non-degeneracy of $C$) with $0 \neq N_3 := [N_1,N_2]$.  By \eqref{E:nn}, $N_3 \in \fa \cap \tau(\fa)$, so extend it to get $\fa \cap \tau(\fa) = \langle N_3, N_4 \rangle$.  Note that $\fn \cong \fn_3 \times \bbC$ and $\cZ(\fn) = \fa \cap \tau(\fa)$.  Since $\fn$ and $\cZ(\fn)$ are $\tau$-stable:
 \begin{itemize}
 \item $\dim(\fe \cap \fn) = 1$: Since $\dim(\fe) = 2$ and $\dim(\fn) = 4$, then $\dim(\fe \cap \fn) \geq 1$.  If $\fe \subset \fn$, then $\ff \subset \fn$, so $C \subset \fn$, which is impossible by non-degeneracy of $C$.

 \item $\fe \cap \cZ(\fn) = 0$: if $0 \neq \fe \cap \cZ(\fn)$, then $0 \neq \ff \cap \cZ(\fn)$, so $\dim(C \cap \cZ(\fn)) \geq 2$ since $\fe \cap \ff = 0$.  Since $\dim(\cZ(\fn)) = 2$, then $\cZ(\fn) \subset C$.  Since $\cZ(\fn)$ is an ideal in $\fg$, then $C$ cannot be non-degenerate.
 \end{itemize}
 Similarly, $\dim(\ff \cap \fn) = 1$ and $\ff \cap \cZ(\fn) = 0$.  
 
 \begin{framed}
 Summarizing, we have the following with $N_2 = \tau(N_1)$ and $N_3 = [N_1, N_2]$:
 \begin{align} \label{E:tau-a}
 \fa = \langle N_1, N_3, N_4 \rangle, \quad \tau(\fa) = \langle N_2, N_3, N_4 \rangle, \quad \cZ(\fn) = \fa \cap \tau(\fa) = \langle N_3, N_4 \rangle.
 \end{align}
Moreover, $\dim(\fe \cap \fn) = \dim(\ff \cap \fn) = 1$, with $\fe \cap \cZ(\fn) = \ff \cap \cZ(\fn) = 0$.
\end{framed}

Let us show that we can assume $\fe \cap \fa \ne 0$, possibly choosing a different 3-dimensional ideal $\fa$ satisfying the above properties.

Since $\dim(\fe \cap \fn) = 1$, write $\fe \cap \fn = \langle \widetilde{N_1} \rangle$ and define $\tilde\fa = \langle \widetilde{N_1}, N_3, N_4 \rangle$.  Since $\fe \cap \cZ(\fn) = 0$, we have $\widetilde{N_1} \not\in \cZ(\fn)$, so $\tilde\fa$ is a 3-dimensional abelian subalgebra, which is clearly an ideal in $\fn$.  Also, $\tau(\tilde\fa) \neq \tilde\fa$ since $\fe \cap \ff = 0$. Let $S \in \fe$ with $S\not\in \fn$, hence $\tau(S) \not\in \fn$ since $\fn$ is $\tau$-stable, and $\fg = \langle S \rangle \op \fn$.  Thus, $\fe = \langle S, \widetilde{N_1} \rangle$ and $\ff = \langle S + v, \widetilde{N_2} \rangle$ for some $v \in \fn$ and $\widetilde{N_2} := \tau(\widetilde{N_1}) \in \ff \cap \fn$.  (We may assume
$v$ has no $\widetilde{N_2}$ component, and redefining $S \mapsto S + c \widetilde{N_1}$, we may in addition assume that $v$ has no $\widetilde{N_1}$-component, i.e.\ $v\in \cZ(\fn)$.)
 Since $\fe$ is a subalgebra and $\fn$ is an ideal in $\fg$, then $[S,\widetilde{N_1}] \in \fe \cap \fn = \langle \widetilde{N_1} \rangle$, so $\tilde\fa$ is an ideal in $\fg$ with $\fe \cap \tilde\fa \neq 0$.  Now replacing $\fa$ with $\tilde\fa$, without loss of generality we can assume that $\fe \cap \fa \neq 0$, and 
 \begin{align}
	\fe = \langle S, N_1 \rangle, \quad
	\ff = \langle S + v, N_2 \rangle,
 \end{align}
where $\{N_1,N_2,N_3,N_4\}$ is a basis of $\fn$ satisfying~\eqref{E:tau-a} and $v\in\cZ(\fn)$.

Since $\fe$ and $\ff$ are subalgebras, and $\fn$ is an ideal, then $a_1 N_1 = [S,N_1]$ and $a_2 N_2 = [S+v,N_2] = [S,N_2]$. Thus,
 \begin{align}
 \begin{split}
 & [N_1, N_2] = N_3, \\
 & [S,N_1] = a_1 N_1, \quad [S,N_2]=a_2 N_2, \\
 & [S,N_3] = (a_1+a_2) N_3, \quad [S,N_4] = a_3 N_3 + a_4 N_4 \in \cZ(\fn).\\
 \end{split}
 \end{align}
 But now an augmentation of $(\fg;\fe,\ff)$ by $\bar\fk = \langle T \rangle$ is given by
 \begin{align}
 & [T,N_1] = N_1, \quad 
 [T,N_2] = -N_2, \quad
 [T,S] = [T,N_3] = [T,N_4] = 0.
 \end{align}
 Thus, $\dim(\ILCsym(\fg;\fe,\ff)) \geq 6$. \qedhere
   \end{enumerate}
 \end{proof}
 \subsection{The $\fa = \tau(\fa)$ case} 
 \label{S:tau-stable} 
  
 Throughout this subsection, we suppose that $\fe \cap \fa \neq 0$ and show that this leads to $\dim(\ILCsym(\fg;\fe,\ff)) \geq 6$.  If $\fe \subset \fa$, then since $\fa$ is $\tau$-stable, we also have $\ff \subset \fa$, hence $C=\fe \op \ff \subset \fa$, which is a contradiction.  Thus, we may assume $\dim(\fe\cap\fa) = 1$, and this implies $\dim(\ff\cap\fa) = 1$.
 Let $\{ X, Y, e_1, e_2, e_3 \}$ be a basis of $\fg$ such that:
 \begin{enumerate}[(i)]
 \item $\fa \cong \bbC^3$ has basis $\{ e_1, e_2, e_3 \}$;
 \item $\fe\cap \fa = \langle e_1 \rangle$ and $\ff\cap \fa = \langle e_3 \rangle$;
 \item $\fe\cap \fa + [\ff,\fe\cap \fa] = \langle e_1, e_2 \rangle$ and $\fe\cap \fa + [\ff,\fe\cap \fa] = \langle e_2, e_3 \rangle$;
 \item $\fe = \langle X, e_1 \rangle$ and $\ff = \langle Y, e_3 \rangle$.
 \end{enumerate}
  Let us clarify (iii). Since $C$ is non-degenerate and $\fa$ is abelian, then $0 \neq [Y,e_1] \mod C$ and so $\dim(\fe\cap \fa + [\ff,\fe\cap \fa]) = 2$.  Applying $\tau$ gives $\dim(\ff\cap \fa+ [\fe,\ff\cap \fa]) = 2$.  These 2-dimensional subspaces of $\fa$ must have 1-dimensional intersection, which we take to be $\langle e_2 \rangle \not\subset C$.

Let $A = \ad_X|_\fa$ and $B = \ad_Y|_\fa$ be represented in the basis $\{ e_1, e_2, e_3 \}$, so:\footnote{We have $i,j=1,2,3$ here and summation is implied over the repeated index $j$.}
 \begin{align} \label{E:XYa-brackets}
 [X,e_i] = A_{ji} e_j, \quad [Y,e_i] = B_{ji} e_j.
 \end{align}
 Note that $[X,e_1] \in \fe \cap \fa$ and $[Y,e_3] \in \ff \cap \fa$, while $[X,e_3] \in \langle e_2, e_3 \rangle$ and $[Y,e_1] \in \langle e_1, e_3 \rangle$ are non-trivial modulo $C$.  Rescaling $X$ and $Y$, we may assume:
 \begin{align} \label{E:AB-gen}
A =\begin{pmatrix} a_{11} & a_{12} & 0 \\ 0 & a_{22} & 1 \\ 0 & a_{32} & a_{33} \end{pmatrix}, \quad
B =\begin{pmatrix} b_{11} & b_{12} & 0 \\ 1 & b_{22} & 0 \\ 0 & b_{32} & b_{33} \end{pmatrix}.
\end{align}
We will exhibit augmentations of $(\fg;\fe,\ff)$ by $\bar\fk = \langle T \rangle$, thereby showing $\dim(\ILCsym(\fg;\fe,\ff)) \geq 6$.
 \subsubsection{$\fg/\fa$ is abelian}
 In this case $[A,B]=0$ and this forces
\begin{align} \label{E:AB}
A =\begin{pmatrix} a_{11} & 0 & 0 \\ 0 & a_{11} & 1 \\ 0 & 0 & a_{33} \end{pmatrix}, \quad
B =\begin{pmatrix} b_{11} & 0 & 0 \\ 1 & b_{33} & 0 \\ 0 & 0 & b_{33} \end{pmatrix}.
\end{align}
Aside from \eqref{E:XYa-brackets}, there is only the bracket $[X,Y] = c_i e_i$.  Define an augmentation of $(\fg;\fe,\ff)$ by $\bar\fk = \langle T \rangle$ (see Definition \ref{D:embedding}) with new (non-trivial) brackets
 \begin{align}
 [T,X]= e_1 - a_{33} T, \quad
 [T,Y]= e_3 - b_{11} T.
 \end{align}
 \subsubsection{$\fg/\fa$ is not abelian}
 We have $0 \not\equiv [X,Y] \equiv \alpha X + \beta Y \mod \fa$.  Requiring $Y \equiv \tau(X) \mod \fa$ forces $\beta = -\overline\alpha$, so necessarily $\alpha\neq 0$.  Rescaling $X$, we normalize $\alpha=1$, so $[X,Y] \equiv X-Y \mod \fa$.
Thus, $[A,B]=A-B$, and we get the following four cases:
 \begin{align}
 \begin{array}{ccccc}
 & A & B\\ \hline
 \mbox{(i)} & \begin{pmatrix} a & 0 & 0\\ 0 & a-1 & 1\\ 0 & 0 & a\end{pmatrix} &
 \begin{pmatrix} a & 0 & 0\\ 1 & a-1 & 0\\ 0 & 0 & a\end{pmatrix}\\
 \mbox{(ii)} & \begin{pmatrix} a+2 & 2 & 0 \\ 0 & a+1 & 1 \\ 0 & 0 & a \end{pmatrix} &
 \begin{pmatrix} a & 0 & 0 \\ 1 & a+1 & 0 \\ 0 & 2 & a+2 \end{pmatrix}\\
 \mbox{(iii)} & \begin{pmatrix} a+2 & 0 & 0\\ 0 & a+1 & 1\\ 0 & 0 & a\end{pmatrix} & 
 \begin{pmatrix} a+2 & 0 & 0\\ 1 & a & 0\\ 0 & 1 & a+1\end{pmatrix}\\
 \mbox{(iv)} & \begin{pmatrix} a+1 & 1 & 0 \\ 0 & a & 1 \\ 0 & 0 & a+2 \end{pmatrix} &
 \begin{pmatrix} a & 0 & 0 \\ 1 & a+1 & 0 \\ 0 & 0 & a+2 \end{pmatrix}
 \end{array}
 \end{align}
 Case (iii) (and similarly, (iv)) does not yield an ASD-ILC triple: the $\tau$-invariant subspace $\langle e_2 \rangle$ is $\ad(\fe)$-invariant, but not $\ad(\ff)$-invariant.  Thus, (iii) and (iv) may be discarded.  For both (i) and (ii), an augmentation of $(\fg;\fe,\ff)$ by $\bar\fk = \langle T \rangle$ is given by
 \begin{align}
 [T,X] = e_1 - (a+1) T, \quad [T,Y] = e_3 - (a+1) T.
 \end{align}

 \section{The non-tubular CR hypersurface with $\saff(2,\R)$-symmetry}
 \label{S:ns21}
 \subsection{Non-tubular and Levi-indefinite}
 
By Theorem \ref{T:NoAb}, there is a unique ASD-ILC triple $(\fg;\fe,\ff)$
on $\fg = \saff(2,\bbC) = \langle H, X, Y, v_1, v_2 \rangle$, see \eqref{E:NS2-EF}.  The fixed-point set of the unique admissible anti-involution $\tau$ from \eqref{E:saff-AI} has $\R$-basis
 \begin{align}
 iH, \quad X+Y, \quad i(X-Y), \quad v_1+v_2, \quad i(v_1 - v_2),
 \end{align}
 and spans $\fg_\bbR := \saff(2,\mathbb{R}) := \fsl(2,\bbR) \ltimes \bbR^2$.  It has 2-dimensional radical, so does not contain a 3-dimensional abelian subalgebra.  The associated CR structure is {\em non-tubular}.  (See Definition \ref{D:tubeCR}.)
 
 Recall that given a CR structure $(M,C,J)$, the complexification $C^\C$ splits into complementary $\pm i$-eigenspaces $C^{1,0}$ and $C^{0,1}$.  Its {\sl Levi form} $\cL$ is the hermitian form given by 
 \[
 \cL : (\xi,\eta) \mapsto [\xi,\overline{\eta}] \,\,\mod C^\C, \quad \forall\xi,\eta \in \Gamma(C^{0,1}).
 \]
 For the CR structure arising from an ASD-ILC triple $(\fg;\fe,\ff)$ and its fixed-point set under an admissible anti-involution, we identify $\fe$ and $\ff$ with $C^{1,0}$ and $C^{0,1}$ respectively, so $\cL$ becomes:
 \[
 \cL: (\xi,\eta) \mapsto [\xi,\tau(\eta)] \,\,\mod \fe \op \ff, \quad \forall \xi,\eta \in \ff.
 \]
 For $\fg = \saff(2,\C)$ with \eqref{E:NS2-EF} and \eqref{E:saff-AI}, take the basis $(\bf_1,\bf_2) = (H-v_2, Y)$, so $\cL$ has components
 \[
 \left(\begin{smallmatrix}
 \cL(\bf_1,\bf_1) & \cL(\bf_1,\bf_2) \\
 \cL(\bf_2,\bf_1) & \cL(\bf_2,\bf_2)
 \end{smallmatrix} \right)
 = \left( \begin{smallmatrix}
 [H-v_2,-H-v_1] & [H-v_2,X]\\
 [Y,-H-v_1] & [Y,X]
 \end{smallmatrix} \right)
 = \left( \begin{smallmatrix}
 v_2 - v_1 & 2X + v_1\\
 -2Y - v_2 & -H
 \end{smallmatrix} \right) 
 \equiv \left( \begin{smallmatrix}
 2H & -H\\
 -H & -H
 \end{smallmatrix} \right)  \mod \fe \op \ff.
 \]
 The coefficient matrix has negative determinant, so $\cL$ has {\em indefinite} signature.
 \subsection{A simple derivation of the model}
 \label{S:SAff2-coord}
 Take the standard action of $\fg = \saff(2,\bbC)$ on $\bbC^2$:
 \begin{align} \label{E:HXYv1v2}
 H = z_1\partial_{z_1} - z_2\partial_{z_2}, \quad X = z_1\partial_{z_2}, \quad Y = z_2\partial_{z_1}, \quad v_1 = \partial_{z_1}, \quad v_2 = \partial_{z_2}.
 \end{align}
 Regarding $(z_1,z_2)$-space $\bbC^2$ as the zeroth jet space $J^0(\bbC,\bbC)$ and using the standard notion of {\sl prolongation} from jet calculus \cite[Thm.4.16]{Olver1995}, we prolong \eqref{E:HXYv1v2} to the first jet space $J^1(\bbC,\bbC)$, i.e.\ $(z_1,z_2,w := z_2')$-space.  Furthermore, induce the joint action on two copies of $J^1(\bbC,\bbC)$, i.e. $(z_1,z_2,w,a_1,a_2,c)$-space.  Using the same vector field labels for their corresponding lifts, we obtain:
 \begin{align} \label{E:saff2-vf}
 \begin{split}
 H &= z_1\partial_{z_1} - z_2\partial_{z_2} - 2w\partial_w + a_1\partial_{a_1} - a_2\partial_{a_2} - 2c\partial_c, \\
 X &= z_1\partial_{z_2} + \partial_w + a_1\partial_{a_2} + \partial_c, \\
 Y &= z_2\partial_{z_1} - w^2\partial_w + a_2\partial_{a_1} - c^2\partial_c, \\
 v_1 &= \partial_{z_1} + \partial_{a_1}, \\
 v_2 &= \partial_{z_2} + \partial_{a_2}.
 \end{split}
 \end{align}
 This prolonged $\fg$-action admits the joint differential invariant (on $w \neq c$):
 \begin{align} \label{E:area}
 \cA := \frac{(z_2 - a_2 - w(z_1-a_1))(z_2 - a_2 - c(z_1-a_1))}{2(w-c)}.
 \end{align}
  
 Consider the complex hypersurfaces $\cA = \lambda$, where $\lambda \in \bbC^\times$. 
Rescalings $(z_1,z_2,w,a_1,a_2,c) \mapsto (\mu z_1, \mu z_2, w,\mu a_1,\mu a_2, c)$ for $\mu \in \C^\times$ allow us to normalize $\lambda$ to $i$ (or any nonzero constant).  Now intersect this hypersurface with the fixed-point set of the anti-involution $(z_1,z_2,w,a_1,a_2,c) \stackrel{\tau}{\mapsto} (\overline{a}_1,\overline{a}_2,\overline{c},\overline{z}_1,\overline{z}_2,\overline{w})$.  This yields an $\saff(2,\bbR)$-invariant CR hypersurface $M^5 \subset \C^3$:
 \begin{align} \label{E:saff2-CR-model}
 w - \overline{w} = -\tfrac{i}{2} (z_2 - \overline{z}_2 - w(z_1 - \overline{z}_1))(z_2 - \overline{z}_2 - \overline{w} (z_1 - \overline{z}_1)),
 \end{align}
 which is the same as \eqref{E:dmt-saff2}.  Explicitly, $\hol(M) \cong \saff(2,\bbR)$ is spanned (as a real Lie algebra) by:
 \begin{align} \label{E:CR-saff2}
 z_1\partial_{z_1} - z_2\partial_{z_2} - 2 w\partial_w, \quad
 z_1\partial_{z_2} + \partial_w, \quad
 z_2\partial_{z_1} - w^2\partial_w, \quad
 \partial_{z_1},\quad
 \partial_{z_2}.
 \end{align}
 (Namely, restrict \eqref{E:saff2-vf} to the fixed-point set of $\tau$ and project to their holomorphic parts.)
 \subsection{An equivalence of models}
 \label{S:Equiv-Loboda}
 On $\bbC^3$, take coordinates $(z_1, z_2, w) = 
(x_1 + iy_1, x_2 + iy_2, u+iv)$.  In this notation, our model \eqref{E:saff2-CR-model} becomes:
\begin{align} \label{E:saff-real}
M_{\mathfrak{saff}}: \qquad
 0 = -v + v^2 y_1^2 + (y_2- y_1 u)^2.
\end{align}
 Under the global biholomorphism of $\bbC^3$ given by 
 \begin{align} \label{E:Lob-bihol}
 (\widetilde{z_1},\widetilde{z_2},\widetilde{w}) = (w,z_1,-z_2+z_1 w),
 \end{align}
 our model in \eqref{E:saff-real} becomes (after dropping tildes):
\begin{align}
M_{\mathsf{Lob}}: \qquad
0 = -y_1 + y_1^2y_2^2 + \left(v-x_2y_1 \right)^2,
\end{align}
which was given in~\cite[pg.50]{Loboda-2020}.  The symmetry algebra of $M_{\mathsf{Lob}}$ was asserted to be 5-dimensional, but the symmetry vector fields for $M_{\mathsf{Lob}}$ were not stated in that work.  Pushing forward our symmetries from \eqref{E:CR-saff2} using \eqref{E:Lob-bihol}, we arrive at the symmetries of $M_{\mathsf{Lob}}$:
 \begin{align}
 \partial_{z_1}, \quad 
 \partial_w, \quad 
 \partial_{z_2} + z_1\partial_{w}, \quad 
  2 z_1\partial_{z_1} - z_2\partial_{z_2} + w\partial_w, \quad
 z_1^2\partial_{z_1} + (w-z_1z_2)\partial_{z_2} + wz_1\partial_w.
 \end{align}
 
 \begin{remark} Using the Levi determinant, we find that our model $M_\saff$ has 4-dimensional Levi degeneracy locus $\{ y_2 - u y_1 = 0, \, v=0 \}$, while that for $M_{\mathsf{Lob}}$ is $\{ y_1=0, \, v=0 \}$.  These loci are mapped to each other under \eqref{E:Lob-bihol}.
 \end{remark}
 \subsection{Related equi-affine geometry} 
 Restricting to the {\em real} setting, we can uncover the geometric meaning of the invariant \eqref{E:area}.  For $(x,y,u,a,b,c) \in \bbR^6 \simeq_{\rm loc} J^1(\R,\R) \times J^1(\R,\R)$, define
 \begin{align}
 \cA = \frac{(y - b - u(x-a))(y - b - c(x-a))}{2(u-c)}.
 \end{align}
 We now give two lovely interpretations for $\cA$.  These are phrased in terms of classical geometric constructions for which invariance under the {\sl planar equi-affine} group $\SAff(2,\R) := \SL(2,\R) \ltimes \R^2$ is manifest, since this group preserves areas and maps lines to lines.
 
 First, fixing $(x,y,u,a,b,c) \in \bbR^6$, consider in $\bbR^2$ the line $L_1$ through the point $(x,y)$ with slope $u$, and the line $L_2$ through $(a,b)$ with slope $c$.  If $u \neq c$, these lines intersect at a unique point $(s,t)$.  Adjoining a third line $L_3$ passing through (distinct) points $(x,y)$ and $(a,b)$ then determines a triangle, and it is a simple exercise to verify that $|\cA|$ is its area.

\begin{figure}[H]
\includegraphics[width=0.8\textwidth]{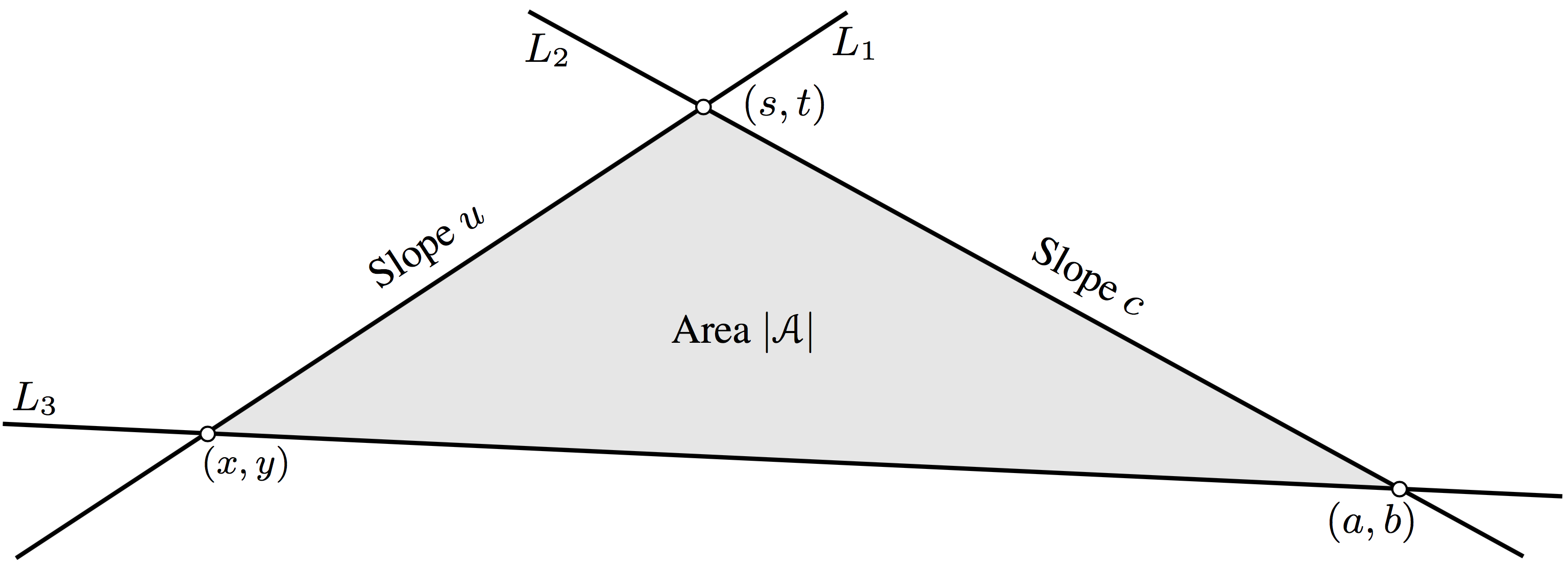}
\caption{Geometric construction of the joint invariant $\mathcal{A}$.}
\end{figure} 

For the second interpretation, let us first recall a classical construction.  Fix $p_0 \in \R^2$ and a line $L_0$ through $p_0$.  Given any line $L$ through $p_0$ that is transverse to $L_0$, consider a hyperbola $\cH$ having asymptotes $L_0$ and $L$.  For any point $p \in \cH$, we can form the:
 \begin{itemize}
 \item {\sl asymptotes-parallelogram} with vertices $p$ and $p_0$ and sides parallel to $L$ and $L_0$.
 \item {\sl tangent-asymptotes-triangle} whose vertices are $p_0$ and the intersection points of tangent line to $\cH$ at $p$ with the asymptotes $L$ and $L_0$.
 \end{itemize}
 Two well-known facts from classical geometry about this construction are:
 \begin{itemize}
 \item One of the diagonals of the asymptotes parallelogram (the one not passing through $p$ and $p_0$) is itself parallel to the tangent line to $\cH$ at $p$.
 \item The area of the asymptotes-parallelogram, which we denote by $\Area(\cH)$, is half that of the tangent-asymptotes-triangle.  Moreover, these areas are {\em constant} for any choice of $p \in \cH$.
 \end{itemize}
 This gives a natural equi-affinely invariant construction: Fix $\cA$ and fix $(a,b,c) \in J^1(\R,\R)$.  The latter determines a point $p_0 := (a,b) \in \R^2$ and line $L_0$ with slope $c$, and we consider the family of all hyperbolas $\cH$ having $L_0$ as one asymptote and having $\Area(\cH) = |\cA|$.  This gives a local foliation of (an open subset of) the plane, as the example below illustrates.  The collection of all such foliations is $\SAff(2,\R)$-invariant.

 \begin{ex} Fix $\cA$.   When $(a,b,c)=(0,0,0)$, solving \eqref{E:area} for $u = y'$ gives the ODE $ y' = \frac{y^2}{xy+2\cA}$.  Rewrite this as $0 = \frac{dx}{y} - \frac{xy+2\cA}{y^3} dy = \frac{dx}{y} - \frac{x}{y^2} dy + \frac{2\cA}{y^3}$, with general solution $\frac{x}{y} + \frac{\cA}{y^2} = \mu \in \R$.  Rearranging gives $y(\mu y-x) = \cA$, which are hyperbolas $\cH_\mu$ with asymptotes $y=0$ and $y=\frac{x}{\mu}$.  A simple exercise shows that $\Area(\cH_\mu) = |\cA|$, independent of $\mu$.
 \end{ex}
 
 \begin{figure}[H]
 \includegraphics[width=0.9\textwidth]{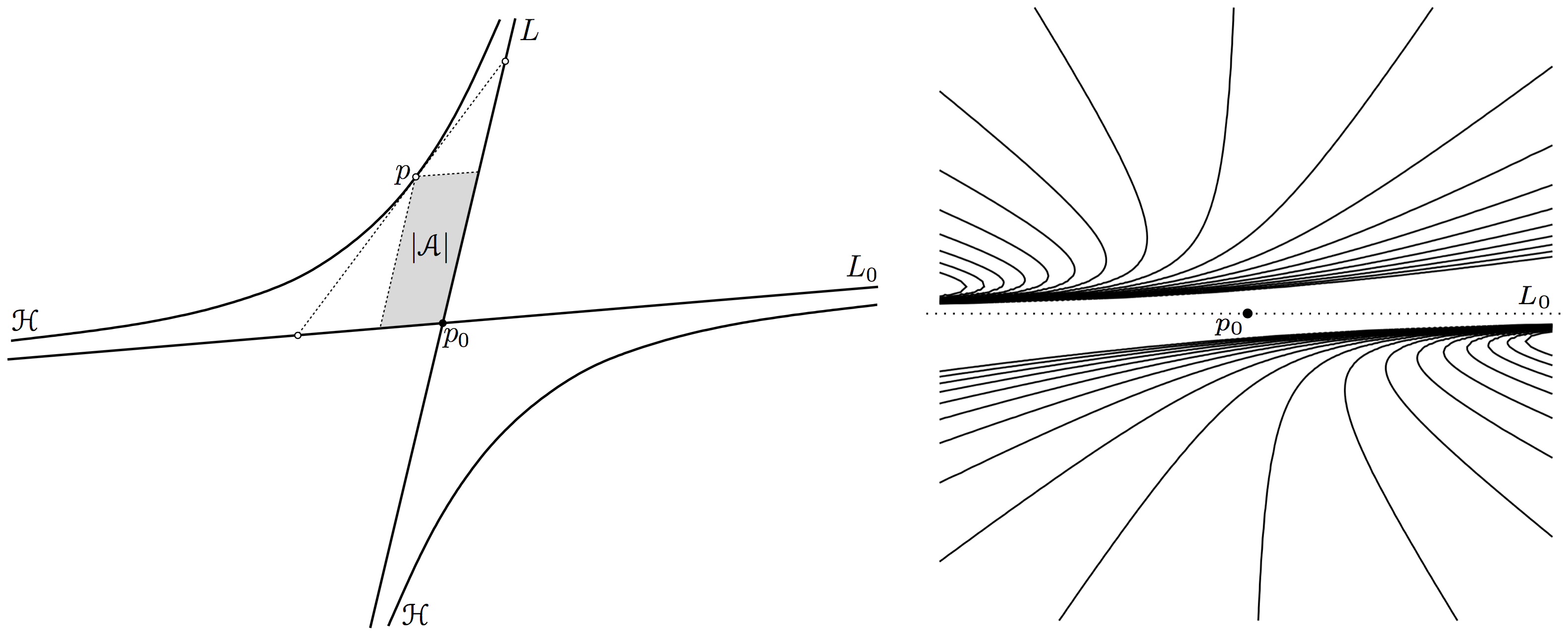}
 \caption{Asymptotes-parallelogram and a foliation by hyperbolas with constant area}
 \end{figure}
 \subsection{Related PDE realization} 
 Let us now describe the compatible, complete system of 2nd order PDEs (\S\ref{S:approach}) that corresponds to the ASD-ILC structure \eqref{E:NS2-EF} with symmetry $\fg = \saff(2,\bbC)$. In other words, we are looking for the equations whose complete solution $w(z_1,z_2)$ is defined by~\eqref{E:saff2-CR-model}. 
By definition, this system of PDEs admits the 5-dimensional Lie algebra of point symmetries \eqref{E:CR-saff2}, which coincides with the lift of $\fg$ to $J^1(\bbC,\bbC)$ as defined in~\S\ref{S:SAff2-coord}.  We identify here $J^1(\bbC,\bbC)$ with $\bbC^3=J^0(\bbC^2,\bbC)$ equipped with coordinates $(z_1,z_2,w)$ and then further prolong $\fg$ to $J^2(\bbC^2,\bbC)$ to determine all $\fg$-invariant complete systems of 2nd order on $w(z_1,z_2)$.

All such systems were computed in the PhD thesis of Hillgarter~\cite{Hillgarter}.  The $\fg$-action lifted to $J^2(\bbC^2,\bbC)$ admits the following three absolute invariants (see p.83 (${\textbf{ip}_{13}}$) and \S4.2.1 of~\cite{Hillgarter}):
\begin{align*}
	I_1 &= \frac{w_{1}^2w_{22}+w_2^2w_{11}-2w_1w_2w_{12}}{(w_1+ww_2)^2},\\
	I_2 &= \frac{w_1w_{12}-w_2w_{11}+w(w_1w_{22}-w_2w_{12})}{(w_1+ww_2)^{5/3}},\\
	I_3 &= \frac{w_{11}+w^2w_{22}-2w_1w_2+2w(w_{12}-w_2^2)}{(w_1+ww_2)^{4/3}}.
\end{align*}
So, any system of 2nd order PDEs admitting point symmetry $\fg$ is (implicitly) given by:
\begin{equation}\label{E:invPDE}
	\{ I_1 = \alpha_1,\quad I_2=\alpha_2,\quad I_3 = \alpha_3 \},
\end{equation}
where $\alpha_i \in \C$.  We now classify those that are {\em compatible}, i.e.\ $E$ from \eqref{E:PDE-EF} is Frobenius-integrable.
 \begin{prop}
 All compatible, complete 2nd order PDE systems $w_{ij} = f_{ij}(z_k,w,w_\ell)$, $1 \leq i,j,k,\ell \leq 2$ that are invariant under \eqref{E:CR-saff2} are equivalent to one of:
 \begin{align}
 \begin{array}{|c|c|c|} \hline
 \begin{cases} 
 I_1 = i,\\
 I_2 = 3(4^{-1/3}) e^{-i\pi/3},\\
 I_3 = -3(4^{1/3}) e^{-i\pi/6}
 \end{cases} & 
 \begin{cases}
 w_{11} = \frac{w_1^2}{(w_2w+w_1)^{2/3}}+\frac{2 w_1^2 w_2}{w_2w+w_1}, \\
 w_{12} = \frac{w_1 w_2}{(w_2w+w_1)^{2/3}} + \frac{2 w_1 w_2^2}{w_2w+w_1}, \\
 w_{22} = \frac{w_2^2}{(w_2w+w_1)^{2/3}} +\frac{2 w_2^3}{w_2w+w_1}
 \end{cases} & 
 \begin{cases}
 w_{11} = \frac{2 w_1^2 w_2}{w_2w+w_1}, \\
 w_{12} = \frac{2 w_1 w_2^2}{w_2w+w_1}, \\
 w_{22} = \frac{2 w_2^3}{w_2w+w_1}
 \end{cases}\\ \hline
 \mbox{Type } \pI{} & \mbox{Type } \pII{} & \mbox{Type } \pIII{} \\ \hline
 \mbox{ASD} & \mbox{not ASD} & \mbox{not ASD}\\ \hline
 \end{array}
 \end{align}
 For the type $\pI{}$ and $\pII{}$ systems above, \eqref{E:CR-saff2} is the full point symmetry algebra, while the type $\pIII{}$ system admits the additional point symmetry $z_1\partial_{z_1} + z_2\partial_{z_2}$ and full point symmetry algebra $\mathfrak{aff}(2,\C)$.
 \end{prop}
 \begin{proof} Solving \eqref{E:invPDE} for $w_{ij}$, we find that \eqref{E:invPDE} is compatible if and only if $3\alpha_1 \alpha_3 = 4\alpha_2^2$ and $9\alpha_1 = \alpha_2 \alpha_3$.  This admits the following solutions:
 \begin{enumerate}
 \item $(\alpha_1, \alpha_2, \alpha_3) = (0,0,\alpha)$:  If $\alpha=0$, we get the third system.  If $\alpha \neq 0$, we normalize it to $\alpha=1$ using the rescaling $(z_1,z_2,w) \mapsto (\lambda z_1, \lambda z_2, w)$, which induces $(I_1,I_2,I_3) \mapsto (\lambda^{-2} I_1, \lambda^{-4/3} I_2, \lambda^{-2/3} I_3)$.  This gives the second system.
 \item $(\alpha_1,\alpha_2,\alpha_3) = (\frac{\alpha^3}{108}, \frac{\alpha^2}{12}, \alpha)$: Evaluating $I_1,I_2,I_3$ on the functions $w(z_1,z_2)$ defined by~\eqref{E:area}, we find that $\alpha = -3(\frac{4}{\cA})^{1/3}$.  (As expected, this does not depend on the parameters $(a_1,a_2,c)$, but only on $\cA$.)  Rescaling as above, we normalize $\cA = i$, which gives the first system.
 \end{enumerate}
Applying~\eqref{E:Q4}, we identify the root types of $\cQ_4$ as indicated.  For the type $\pI{}$ and $\pII{}$ cases, \eqref{E:RootTypeSym} confirms 5-dimensional symmetry, while there is the additional indicated symmetry for type $\pIII{}$ case.  (From \cite[Table 2]{ILC}, this is a realization of model $\pIII{.6-2}$.)  From Proposition \ref{P:saff} and Example \ref{X:saff}, an $\saff(2,\C)$-invariant ASD-ILC structure must be of type $\pI{}$.
 \end{proof}

 The type $\pI{}$ realization above is the desired PDE system with associated CR hypersurface \eqref{E:saff2-CR-model}.

 \section{Simply-transitive tubular hypersurfaces}
 \label{S:tubular}
 \subsection{From homogeneous tubes to algebraic data}
 Given a real affine hypersurface $\cS \subset \R^{n+1}$, we discussed in \S\ref{S:intro} its associated tubular CR hypersurface $M_\cS \subset \C^{n+1}$, and its complexification $M^c_\cS \subset \C^{n+1} \times \C^{n+1}$ is the associated {\sl tubular ILC hypersurface}.  (We recover $M_\cS$ as the fixed-point set
of the anti-involution $\tau(z,a) = (\overline{a},\overline{z})$ restricted to $M^c_\cS$.)  The symmetry algebra $\sym(M^c_\cS)$ is the complex Lie algebra consists of all holomorphic vector fields $X = \xi^k(z) \partial_{z_k} + \sigma^k(a) \partial_{a_k} \in \fX(\C^{n+1}) \times \fX(\C^{n+1})$ that are everywhere tangent to $M^c_\cS$.  The {\sl affine symmetry algebra} $\aff(\cS)$ consists of those {\sl affine vector fields} $\bS = (A_{k\ell} x_\ell + b_k) \partial_{x_k}$, for $A_{k\ell}, b_k \in \R$, that are everywhere tangent to $\cS$.  Any $\bS \in \aff(\cS)$ induces symmetries of $\bS^{\CR}$ of $M_\cS$ and $\bS^{\LC}$ of $M^c_\cS$ as indicated below.  We respectively denote the induced real and complex Lie algebras by $\aff(\cS)^{\CR} \subset \hol(M_\cS)$ and $\aff(\cS)^{\LC} \subset \sym(M^c_\cS)$, and it is clear that $\aff(\cS)^{\LC} \cong \aff(\cS)^{\CR} \otimes_\R \C \cong \aff(\cS) \otimes_\R \C$.
 
 \begin{figure}[H]
 \begin{center}
 \begin{tabular}{c}
\begin{tabular}{|c|} \hline
Real affine hypersurface\\ \hline
$\cS = \{ x : \cF(x) = 0 \} \subset \R^{n+1}$, \quad $d\cF \neq 0$ on $\cS$;\\
Real affine symmetry $\bS = (A_{k\ell} x_\ell + b_k) \partial_{x_k} \in \aff(\cS)$\\ \hline
\end{tabular}\\
\\
\begin{tabular}{ccc}
\begin{tabular}{|c|c|} \hline
Tubular CR hypersurface\\ \hline
$M_\cS = \{ z : \cF(\Re z) = 0 \} \subset \C^{n+1}$;\\ 
 $i\partial_{z_1},...,i\partial_{z_{n+1}} \in \hol(M_\cS)$, \\
 $\bS^{\CR} := (A_{k\ell} z_\ell + b_k) \partial_{z_k} \in \aff(\cS)^{\operatorname{cr}}$\\ \hline
\end{tabular} &
\begin{tabular}{|c|} \hline
Tubular ILC hypersurface\\ \hline
$M^c_\cS = \{ (z,a) : \cF(\frac{z+a}{2}) = 0 \} \subset \C^{n+1} \times \C^{n+1}$;\\ 
 $\partial_{z_1} - \partial_{a_1},..., \partial_{z_{n+1}} - \partial_{a_{n+1}} \in \sym(M^c_\cS)$, \\
 $\bS^{\operatorname{lc}} := (A_{k\ell} z_\ell + b_k) \partial_{z_k} + (A_{k\ell} a_\ell + b_k) \partial_{a_k} \in \aff(\cS)^{\LC}$\\ \hline
\end{tabular}
\end{tabular}
\end{tabular}
 \end{center}
 \end{figure}

 \begin{remark} \label{R:CAff-hyp}
 Any {\em complex} affine hypersurface $\cS \subset \C^{n+1}$ also induces a tubular ILC hypersurface $M^c_\cS \subset \C^{n+1} \times \C^{n+1}$ via the same prescription above.
 \end{remark}

For $M^c_\cS$, note that $\fa = \langle \partial_{z_1} - \partial_{a_1}, \ldots, \partial_{z_{n+1}} - \partial_{a_{n+1}} \rangle$ is an $(n+1)$-dimensional abelian Lie algebra $\fa \subset \fg := \sym(M^c_\cS)$ that is transverse to $E$ and $F$ (as defined in \S\ref{S:approach}), so we are naturally led to the following algebraic data for any holomorphically homogeneous tube:

\begin{defn} \label{D:tubeCR}
	A {\em tubular CR realization} for an ILC quadruple $(\fg,\fk;\fe,\ff)$ in dimension $\dim(\fg / \fk) = 2n+1$ is a pair $(\fa,\tau)$, where
	\begin{enumerate}
		\item[(T.1)] $\fa \subset \fg$ is an $(n+1)$-dimensional abelian subalgebra;
		\item[(T.2)] $\fe \cap \fa = \ff \cap \fa = 0$.
		\item[(T.3)] $\tau$ is an admissible anti-involution of $(\fg,\fk;\fe,\ff)$ that preserves $\fa$.
	\end{enumerate}
\end{defn}

 Conversely, given such data as above, we integrate $(\fg,\fk)$ to a (local) homogeneous space $N = G/K$ with $G$-invariant distributions $E,F$.  Since $C = E \op F$ is non-degenerate, then all symmetries of the ILC structure $(N;E,F)$ are in 1-1 correspondence with their projection by $d\pi_1$ or $d\pi_2$.  (We refer to the double fibration \eqref{E:2-fib}.)  This implies that the direct product of $\pi_1$ and $\pi_2$ gives a local embedding $N \to N/E \times N/F$ (with codomain being locally $\C^{n+1} \times \C^{n+1}$).  As $\fa$ is abelian, we can identify it with $\C^{n+1}$, with the anti-involution $\tau$ acting on it as $w\mapsto -\bar{w}$ (in the standard basis $\bb$ on $\C^{n+1}$). Let $A\subset G$ be the corresponding subgroup, which can also be locally identified with $\bbC^{n+1}$ equipped with the same anti-involution. Due to (T.1) and (T.2) the action of $A$ on both $N/E$ and $N/F$ is (locally) simply transitive. So, we can identify both $N/E$ and $N/F$ with some open subsets of $\C^{n+1}$, on which we introduce local coordinates $z$ and $a$ relative to $\bb$ and $-\bb$ respectively.  Hence, $\fa = \langle \partial_{z_k} - \partial_{a_k} \rangle$.
 
 Since $\tau$ swaps $\ff$ and $\fe$, it extends to the direct product $N/E \times N/F$ as $\widetilde\tau(z,a)=(\bar a, \bar z)$.
 The embedding $N \inj N/E \times N/F \cong_{{\rm loc}} \C^{n+1} \times \C^{n+1}$ is given by a single complex analytic equation $\Phi(z,a) = 0$.  Invariance of $N$ under $\fa$ forces $N = \{ (z,a) : \cF((z+a)/2) = 0 \}$.  Finally, taking the slice of $\bbC^{n+1}\times\bbC^{n+1}$ defined as a fixed-point set of $\widetilde\tau$, we arrive at the tubular hypersurface $M_{\cS}=\{z\colon \cF(\Re z)=0\}\subset \bbC^{n+1}$, where $\cF$ is now real-valued.  It is a tube over the base $\cS = \{ x :\cF(x) = 0 \} \subset \R^{n+1}$.

 \begin{lem}
 $\fn(\fa) / \fa \cong \aff(\cS) \otimes_\R \C$.
 \end{lem}

 \begin{proof}
  Clearly, $\tspan_\C\{ \bS^{\LC} : \bS \in \aff(\cS) \} \oplus \fa \subset \fn(\fa)$.  Conversely, if $X = \xi^k(z) \partial_{z_k} + \sigma^k(a) \partial_{a_k}$ normalizes $\fa = \langle \partial_{z_1} - \partial_{a_1}, \ldots, \partial_{z_{n+1}} - \partial_{a_{n+1}} \rangle$, then $X = (A_{k\ell} z_\ell + b_k) \partial_{z_k} + (A_{k\ell} a_\ell + c_k) \partial_{a_k}$ for some $A_{k\ell}, b_k, c_k \in \bbC$.    Adding $(\frac{c_k - b_k}{2})(\partial_{z_k} - \partial_{a_k}) \in \fa$, we may assume that $b_k = c_k$.  Since $\fa$ is stable under $d\tau$ (where $\tau(z,a) = (\bar{a},\bar{z})$), then so is $\fn(\fa)$.  Since $\tau^2 = \id$, we can decompose $\fn(\fa)$ into $\pm 1$ eigenspaces for $d\tau$.  Modulo $\fa$, the $+1$ eigenspace consists of $X = (A_{k\ell} z_\ell + b_k) \partial_{z_k} + (A_{k\ell} a_\ell + b_k) \partial_{a_k} \in \fn(\fa)$ with $A_{k\ell}, b_k \in \R$, hence $X = \bS^{\LC}$, where $\bS = (A_{k\ell} x_\ell + b_k) \partial_{x_k} \in \aff(\cS)$.  The $-1$ eigenspace consists of similar vector fields, but with $A_{k\ell},b_k \in i\R$.  Thus, $\fn(\fa) \equiv \tspan_\C\{ \bS^{\LC} : \bS \in \aff(\cS) \} \mod \fa$, which implies the claim.
 \end{proof}

\begin{cor} \label{C:tubes}
	Let $M^5 \subset \C^3$ be a holomorphically simply-transitive, Levi non-degenerate hypersurface  with $\hol(M)$ containing a 3-dimensional abelian ideal.  Then $M$ is a tube on an affinely simply-transitive base.
\end{cor}

\begin{proof} By \eqref{E:sym-dim}, the induced ILC structure on $M^c$ is simply-transitive, so can be encoded by an ASD-ILC triple $(\fg;\fe,\ff)$, where $\fg = \sym(M^c) = \hol(M) \otimes_\R \C$ is 5-dimensional and admits some admissible anti-involution $\tau$.  By hypothesis, $\hol(M)$ contains a 3-dimensional abelian ideal, so there exists a 3-dimensional abelian ideal $\fa \subset \fg$.

	Applying Theorem \ref{T:NTab}, we get $\fa = \tau(\fa)$ and $\fe \cap \fa = \ff \cap \fa = 0$.  Thus, $(\fa,\tau)$ is a tubular CR realization for the ILC triple $(\fg;\fe,\ff)$.  Since $\fa$ is an ideal in $\fg$, then $\fg/\fa = \fn(\fa)/\fa \cong \aff(\cS) \otimes_\R \C$ for some base $\cS$ as constructed above.  As $\hol(M)$ is transitive on $M$, we see that the projection $\hol(M)$ onto $\cS$ is also transitive.  
Thus, $\cS$ is affinely simply-transitive.
\end{proof}
 
 Given $p \in \C^{n+1}$, there is a natural isomorphism of the Lie algebra of all (real or complex) affine vector fields with $\aff(n+1,\C) := \fgl(n+1,\C) \ltimes \C^{n+1}$, via
 \begin{align} \label{E:aff-Ab}
 (A_{k\ell} (z_\ell - p_\ell) + b_k) \partial_{z_k} \quad\mapsto\quad (A,b),
 \end{align}
 for which $A$ is the {\sl linear part at $p$}, and $b$ is the {\sl translational part}.  Recall that conjugation by $P \in \GL(n+1,\C) \subset \Aff(n+1,\C)$ induces the action $(A,b) \mapsto (PAP^{-1},Pb)$.  Finally, $\aff(n+1,\C)$ has a unique abelian ideal consisting of translations $\langle \partial_{x_k} \rangle \cong \C^{n+1}$.
	
 \begin{prop} \label{P:AffToILC} Let $\cS \subset \R^{n+1}$ be an affinely homogeneous hypersurface with non-degenerate 2nd fundamental form.  Then the tubular ILC hypersurface $M_\cS^c \subset \C^{n+1} \times \C^{n+1}$ is homogeneous and encoded by an ILC quadruple $(\fg,\fk;\fe,\ff)$, given for any $p \in \cS$ by
 \begin{align} \label{E:AffToILC}
 \fe := \aff(\cS) \otimes_\R \C, \quad \fg := \fe \ltimes \C^{n+1}, \quad \ff := \{ Y \in \fg : Y|_p = 0 \}, \quad \fk := \fe \cap \ff.
 \end{align}
 \end{prop}
	
 \begin{proof} Since $\cS$ is affinely homogeneous, then $M^c_\cS$ is homogeneous, with $\sym(M^c_\cS)$ containing\begin{align} \label{E:affS-ext}
\fg = \aff(\cS)^{\LC} \op \tspan_\bbC\{ \partial_{z_1} - \partial_{a_1}, \ldots, \partial_{z_{n+1}} - \partial_{a_{n+1}} \},
\end{align}
 which is transitive on $M^c_\cS$.  Given $p \in \cS$, we have $(p,p) \in M^c_\cS$ and
 \begin{align}
 \fe = \{ Y \in \fg : d\pi_2|_{(p,p)}(Y) = 0 \}, \quad
 \ff = \{ Y \in \fg : d\pi_1|_{(p,p)}(Y) = 0 \}, \quad \fk = \fe \cap \ff,
 \end{align}
 in terms of the double fibration \eqref{E:2-fib}.
 Explicitly, let $X := (A_{k\ell} x_\ell + b_k) \partial_{x_k} \in \aff(\cS)$ and $T_{X,p} := (A_{k\ell} p_\ell + b_k) (\partial_{z_k} - \partial_{a_k}) \in \sym(M_\cS^c)$, where $p = (p_1,\ldots,p_{n+1}) \in \cS$.  Consider
\begin{align}
X^{\LC} + T_{X,p} &= (A_{k\ell} (z_\ell + p_\ell) + 2b_k) \partial_{z_k} + A_{k\ell} (a_\ell - p_\ell) \partial_{a_k}, \label{E:XT1}\\
X^{\LC} - T_{X,p} &= A_{k\ell} (z_\ell - p_\ell) \partial_{z_k} + (A_{k\ell} (a_\ell + p_\ell) + 2b_k) \partial_{a_k}. \label{E:XT2}
\end{align}
 Clearly, $\fe = \tspan_\C\{ X^{\LC} + T_{X,p} : X \in \aff(\cS) \}$, while $\ff = \tspan_\C\{ X^{\LC} - T_{X,p} : X \in \aff(\cS) \}$. 
 
 Since $C = E \op F$ is non-degenerate, then all elements of $\sym(M^c)$ are in 1-1 correspondence with their projection by either $d\pi_1$ or $d\pi_2$.   Focusing on their $d\pi_1$ projections, it is clear that $(d\pi_1(\fg),d\pi_1(\ff))$ agree with $(\fg,\ff)$ in \eqref{E:AffToILC}.  Letting $\mathcal{D} \in \Aff(n+1,\C)$ be the dilation centered at $p$ by a factor $\half$, define $(\bar\fg,\bar\fk;\bar\fe,\bar\ff)$ be the (isomorphic) projection of $(\fg,\fk;\fe,\ff)$ by $d\cD \circ d\pi_1$.  Let us view this in terms of \eqref{E:aff-Ab}.  Letting $v = Ap+b$ and $D = \half\, \id$, \eqref{E:XT1} and \eqref{E:XT2} become:
 \begin{align}
 (A,2v) \mapsto (DAD^{-1},2Dv) = (A,v), \qquad
 (A,0) \mapsto (DAD^{-1},0) = (A,0).
 \end{align}
 Via \eqref{E:aff-Ab}, the former is $(A_{k\ell} z_\ell + b_k) \partial_{z_k}$.  Thus, after dropping bars, $(\bar\fg,\bar\fk;\bar\fe,\bar\ff)$ agrees with \eqref{E:AffToILC}.
 \end{proof}

 Note that $\ff \subset \fg$ is the {\sl isotropy subalgebra} at $p$.  Using \eqref{E:Phi-ST} and Proposition \ref{P:AffToILC}, the quartic $\cQ_4$ can be efficiently computed for tubes  $M^c_\cS$ over affinely simply-transitive $\cS$ (see Table \ref{F:tubular-LC-Q4}).
 
 \subsection{Tubes on affinely simply-transitive surfaces}
 \label{S:tubes}

We finally address the tubular simply-transitive Levi non-degenerate classification.  From our work above, these can all be described as tubes on an {\em affinely simply-transitive} base\footnote{Several holomorphically multiply-transitive tubes have base surface that is affinely {\em inhomogeneous} \cite[Tables 7 \& 8]{CRhom}.}.  For the latter, we will use the DKR classification~{\cite{AffHom}} of surfaces in {\em real} affine 3-space and proceed with the initial steps described in \S \ref{S:main-result}.

From the DKR list, we begin by excluding those surfaces whose associated tube already {\em explicitly} appears in the multiply-transitive classification \cite{CRhom}.  In Table \ref{F:MT-CR}, these known tubes are indicated with their ILC quartic types and symmetry dimensions, keeping in mind \eqref{E:sym-dim}.  (The additional hyphenated suffix, e.g. $\sfD$.6-1 and $\sfD$.6-2, indicates labelling of different families derived from \cite{ILC}.)  Finally, we restrict to affinely simply-transitive surfaces with non-degenerate Hessians.  This excludes quadrics, cylinders, and the Cayley surface $u = x_1 x_2 - \frac{x_1^3}{3}$.  (The last of these admits the affine symmetries $x_1\partial_{x_1} + 2x_2\partial_{x_2} + 3u\partial_u$, $\partial_{x_1} + x_1\partial_{x_2} + x_2\partial_u$, and $\partial_{x_2} + x_1\partial_u$.)

 \begin{footnotesize}
 \begin{table}[H]
 \[
 \begin{array}{|c|c|l|l|cccc} \hline
 \mbox{DKR label} & \mbox{Non-degenerate real affine surface} & \mbox{ILC Classification \cite{ILC}} \\ \hline\hline
 (3) & \begin{array}{c} 
 u = \ln(x_1) + \alpha \ln(x_2)\\
 {\scriptstyle(\alpha\neq 0)}
 \end{array} &
 \begin{tabular}{@{}l@{}} 
 \pD{.7}: $\alpha\neq 0,-1$;\\
 \pO{.15}: $\alpha=-1$
 \end{tabular} \\ \hline
 (4) & u = \alpha\arg(ix_1+x_2) + \ln(x_1^2+x_2^2) & \pD{.7}\\
 & u = \arg(ix_1 + x_2) & \pO{.15}\\ \hline
 (7) & u = x_2^2 + \epsilon e^{x_1} & \pO{.15} \\ \hline
 (8) & 
 \begin{array}{c} 
 u = x_2^2 + \epsilon x_1^\alpha\\
 {\scriptstyle(\alpha\neq 0,1)}
 \end{array} &
 \begin{tabular}{@{}l@{}}
 \pD{.6-2}: $\alpha\neq 0,1,2$;\\
 \pO{.15}: $\alpha=2$\\
 \end{tabular} \\ \hline
 (9) & u = x_2^2 + \epsilon \ln(x_1) & \pD{.7} \\\hline
 (10) & u = x_2^2 + \epsilon x_1\ln(x_1) & \pD{.6-2} \\\hline
 (11) & u = x_1 x_2 + e^{x_1} & \pN{.6-2} \\\hline
 (12) & u = x_1 x_2 + x_1^\alpha &
 \begin{tabular}{@{}l@{}}
 \pN{.6-1}: $\alpha \neq 0,1,2,3,4$;\\
 \pN{.8}: $\alpha =4$;\\
 \pO{.15}: $\alpha=0,1,2,3$
 \end{tabular} \\\hline
 (13) & u = x_1 x_2 + \ln(x_1) & \pN{.6-1} \\\hline
 (14) & u = x_1 x_2 + x_1\ln(x_1) & \pN{.7-2} \\\hline
 (15) & u = x_1 x_2 + x_1^2\ln(x_1) & \pN{.6-1} \\\hline
 (17) & x_1 u = x_2^2 + \epsilon x_1\ln(x_1) & \pD{.6-1} \\\hline
 \end{array}
 \]
 \caption{Affinely simply-transitive surfaces with holomorphically multiply-transitive associated tubes.  Parameters $\alpha \in \R$ and $\epsilon = \pm 1$.}
 \label{F:MT-CR}
 \end{table}
 \end{footnotesize}

 \begin{remark} \label{R:lim-case} Family (4) was originally stated in \cite{AffHom} as $u = \alpha\arg(ix_1+x_2) + \beta\ln(x_1^2+x_2^2)$.  Scaling $u$ yields the two cases in Table \ref{F:MT-CR}, the first of which explicitly appears in \cite{CRhom}.  
 The tube $M$ over $u = \arg(ix_1+x_2)$ is mapped to the hyperquadric $\Im \widetilde{w} = |\widetilde{z}_1|^2 - |\widetilde{z}_2|^2$ by
 \begin{align}
 (\widetilde{z}_1,\widetilde{z}_2,\widetilde{w}) = \left( 
 \frac{1}{\sqrt{2}} \Big( e^{iw} - \frac{z_1-iz_2}{4} \Big), 
 \frac{1}{\sqrt{2}} \Big( e^{iw} + \frac{z_1-iz_2}{4} \Big),
 e^{iw}\Big(\frac{iz_1 - z_2}{2}\Big)  \right).
 \end{align}
 Thus, $\dim\,\hol(M) = 15$ and $M$ is flat.  The above was derived from \cite[Thm.6.1(6) \& (6.69)]{Isaev2011}.
 
 Model (16) when $\alpha=0$ gives the quadric $x_1 u = x_2^2 + \epsilon$, with affine symmetries:
 $x_1 \partial_{x_1} - u \partial_u$,
 $2x_2\partial_{x_1} + u\partial_{x_2}$, and
 $x_1 \partial_{x_2} + 2 x_2 \partial_u$.
 Its associated tube admits $\fso(1,2) \ltimes \bbR^3$ symmetry. 
 \end{remark}
 
   All remaining surfaces\footnote{The enumeration (1), (2), (5), (6), (16), (18) from \cite{AffHom} has been re-enumerated as $\sfT1$--$\sfT6$ here.} are given in Table \ref{F:tubular-final}, and their affine symmetries $\bS,\bT$ are given in Table \ref{F:tubular-LC-Q4}.  The associated tubes $M$ admit symmetries $\bS^{\CR}, \bT^{\CR}, i\partial_{z_1}, i\partial_{z_2}, i\partial_w \in \hol(M)$, so $\dim\,\hol(M) \geq 5$.  In Table \ref{F:tubular-LC-Q4}, we compute $\cQ_4$ using \eqref{E:Phi-ST} and Proposition \ref{P:AffToILC}, and classify its root type.  (For details, we refer to a {\tt Maple} file in our {\tt arXiv} submission.)  By \eqref{E:RootTypeSym}, those of type $\pI{}$ and $\pII{}$ are confirmed to have $\dim\,\hol(M) = 5$, so only the type $\pD{}$ and $\pN{}$ cases remain.  We used two methods to computationally confirm that $\dim\,\hol(M) = 5$ for these remaining cases: (i) {\sl PDE point symmetries} (\S\ref{S:PDE-method}), and (ii) {\sl power series} (\S\ref{S:ps-method}).
  

\begin{landscape}\centering
\begin{footnotesize}
\begin{table}[H]
\centering
\[
 \begin{array}{|c|@{}c@{}|@{}l@{}|@{}c@{}|l|c|} \hline
 & \begin{tabular}{c}
 Generic point \\
 $(x_1,x_2,u)$\\
 on surface
 \end{tabular} & \begin{tabular}{l} Affine symmetries $\bS,\bT$, isotropy fields $\widetilde\bS,\widetilde\bT$,\\ and LC-adapted framing $\{ \be_1, \be_2, \bf_1, \bf_2 \}$\end{tabular} & \mbox{ILC Quartic } \cQ_4 & \mbox{Root type of $\cQ_4$}\\ \hline\hline
 \sfT1 & (1,1,1) &
 \begin{array}{l}
 \begin{array}{@{}l@{}}
 \bS = x_1 \partial_{x_1} + \alpha u \partial_u,\\
 \bT = x_2 \partial_{x_2} + \beta u \partial_u
 \end{array} \quad
 \begin{array}{@{}l@{}}
 \widetilde\bS = \bS - \partial_{x_1} - \alpha\partial_u, \\
 \widetilde\bT = \bT - \partial_{x_2} - \beta\partial_u
 \end{array}\\ \hline
 \begin{array}{@{}l}
 \be_1 = \alpha \bT, \quad
 \be_2 = \alpha\beta\bS - \alpha(\alpha-1) \bT,\\
 \bf_1 = \frac{\alpha+\beta-1}{\alpha} \widetilde\bS, \quad
 \bf_2 = \widetilde\bT - \frac{\beta-1}{\alpha} \widetilde\bS
 \end{array}
 \end{array} & 
 \begin{array}{l}
 (\alpha-1)(\alpha+2\beta-1)(\alpha+\beta-1) t^4\\
 \quad + 4(\alpha-1)(\beta-1)(\alpha+\beta-1) t^3\\
 \quad + 2(\alpha-1)(\beta-1)(\alpha+\beta-3) t^2\\
 \quad - 4(\alpha-1)(\beta-1) t\\
 \quad + \frac{(\beta-1)(\beta+2\alpha-1)}{\alpha+\beta-1}
 \end{array} & 
 \begin{tabular}{@{}l}
 Let $S(\alpha,\beta) := (\alpha-1)(\beta-1)(\alpha+\beta)$.  Then:\\\\
 \pI{}: $S(\alpha,\beta)\left(S(\alpha,\beta)+8\alpha\beta\right) \neq 0$\\
 \pII{}: $S(\alpha,\beta) = -8\alpha\beta$, excl. $\pD{}$\\
 \pD{}: $(\alpha,\beta) = (-1,-1)$\\
 \pN{}: Exactly one of $\alpha=1$ or $\beta =1$ or $\beta = -\alpha$\\
 \pO{}: $(\alpha,\beta) \in \{ (1,1),(1,-1),(-1,1) \}$
 \end{tabular} 
 \\ \hline
 \sfT2 & (1,0,1) &
 \begin{array}{l}
 \begin{array}{@{}l@{}}
 \bS = x_1 \partial_{x_1} + x_2 \partial_{x_2} + 2 \alpha u \partial_u,\\
 \bT = x_2 \partial_{x_1} - x_1 \partial_{x_2} - \beta u\partial_u
 \end{array} \quad 
 \begin{array}{l}
 \widetilde\bS = \bS - \partial_{x_1} - 2\alpha\partial_u, \\
 \widetilde\bT = \bT + \partial_{x_2} + \beta\partial_u
 \end{array}\\ \hline
 \alpha \neq 0:\\
 \begin{array}{l}
 \be_1 = \frac{1}{2\alpha-1} \bS, \quad
 \be_2 = 2\alpha\bT + \beta \bS,\\
 \bf_1 = (4\alpha^2+\beta^2) \widetilde\bS, \quad
 \bf_2 = 2\alpha\widetilde\bT + \beta\widetilde\bS
 \end{array}\\
 \alpha=0:\\
 \begin{array}{l}
 \be_1 = \bT, \quad
 \be_2 = \bS, \quad
 \bf_1 = \widetilde\bS, \quad
 \bf_2 = \widetilde\bT - \beta\widetilde\bS
 \end{array}
 \end{array} &
 \begin{array}{@{}l@{}}
 \begin{array}{l}
 -2(\alpha-1)(4\alpha^2+\beta^2)^2 t^4 - 8\beta(4\alpha^2+\beta^2) t^3 \\
 \quad+ \frac{4\left(4\alpha^2 (2 \alpha+1)(\alpha-1) + \beta^2 (2 \alpha^2+3 \alpha-3)\right)}{2\alpha-1} t^2\\
 \quad+ \frac{8\beta}{2\alpha-1} t - \frac{2(\alpha-1)}{(2\alpha-1)^2}, \quad\mbox{for}\quad \alpha \neq 0
 \end{array}\\ \\ 
 \begin{array}{l}
 \beta(\beta^2+4),\quad\mbox{for}\quad \alpha = 0
 \end{array}
 \end{array} & 
 \begin{tabular}{@{}l}
 \pI{}: $4\alpha(\alpha+1)^2+(\alpha+4)\beta^2 \neq 0$, \mbox{ excl. \pN{} \& \pO{}}\\
 \pII{}: $4\alpha(\alpha+1)^2+(\alpha+4)\beta^2 = 0$, \mbox{ excl. \pD{}}\\
 \pD{}: $(\alpha,\beta) = (-1,0)$\\
 \pN{}: $\alpha = 0$, $\beta \neq 0$\\
 \pO{}: $(\alpha,\beta) = (1,0)$
 \end{tabular} 
 \\ \hline
 \sfT3 & (1,1,0) &
 \begin{array}{l}
 \begin{array}{@{}l@{}}
 \bS = x_1 \partial_{x_1} - \alpha x_2 \partial_{x_2} + u \partial_u,\\
 \bT = x_2 \partial_{x_2} + x_1 \partial_u
 \end{array} \quad
 \begin{array}{@{}l@{}}
 \widetilde\bS = \bS - \partial_{x_1} + \alpha\partial_{x_2},\\
 \widetilde\bT = \bT - \partial_{x_2} - \partial_u 
 \end{array}\\ \hline
 \begin{array}{@{}l}
 \be_1 = \frac{1}{1+\alpha} \bS + \bT, \quad
 \be_2 = \bT,\\
 \bf_1 = \widetilde\bS + (1+\alpha) \widetilde\bT, \quad
 \bf_2 = -\widetilde\bT,
 \end{array}
 \end{array} &
 -t^4 - 4t^3 - \frac{2(\alpha+3)}{\alpha+1} t^2 - \frac{4}{\alpha+1} t -\frac{1}{(\alpha+1)^2} & 
  \begin{tabular}{@{}l@{}}
 \pI{}: $\alpha \neq -1,0,8$\\
 \pII{}: $\alpha=8$ \\
 \pN{}: $\alpha=0$
 \end{tabular}
 \\ \hline
 \sfT4 & (0,\alpha^{-1/3},1) &
 \begin{array}{l}
 \begin{array}{@{}l}
 \bS = x_1 \partial_{x_1} + 2 x_2 \partial_{x_2} + 3 u \partial_u,\\
 \bT = \partial_{x_1} + x_1 \partial_{x_2} + x_2 \partial_u
 \end{array} \,
 \begin{array}{@{}l}
 \widetilde\bS = \bS - 2\alpha^{-1/3}\partial_{x_2} - 3\partial_u, \\
 \widetilde\bT = \bT - \partial_{x_1} - \alpha^{-1/3} \partial_u 
 \end{array}\\ \hline
 \begin{array}{@{}l}
 \be_1 = \bS+\frac{4}{3}\alpha^{-2/3} \bT, \quad
 \be_2 = \alpha^{-2/3} \bT,\\
 \bf_1 = \widetilde\bS + \frac{4}{3} \alpha^{-2/3} \widetilde\bT, \quad
 \bf_2 = -\frac{2}{9} (9\alpha+8)\alpha^{-2/3} \widetilde\bT,\\
 \end{array}
 \end{array}
 &
 \frac{4\left(9 t^2+(9\alpha+8)(3 t+2)\right)^2}{27\alpha(9\alpha+8)} & 
 \begin{tabular}{@{}l}
 \pD{}: $\alpha \neq 0,-\frac{8}{9}$\\
 \end{tabular}
 \\ \hline
 \sfT5 & (1,0,\epsilon)& 
 \begin{array}{l} 
 \begin{array}{@{}l}
 \bS = x_1 \partial_{x_1} + \frac{\alpha}{2} x_2 \partial_{x_2} + (\alpha-1) u \partial_u,\\
 \bT = x_1 \partial_{x_2} + 2 x_2 \partial_u
 \end{array}\\ \hline
 \begin{array}{@{}l}
 \be_1 = \frac{1}{(\alpha-1)(\alpha-2)} \bS, \quad
 \be_2 = \frac{\epsilon}{2} \bT,\\
 \bf_1 = \widetilde\bS := \bS - \partial_{x_1} - \epsilon(\alpha-1)\partial_u, \quad
 \bf_2 = \widetilde\bT := \bT - \partial_{x_2}
 \end{array}
 \end{array}
 &
\frac{\epsilon}{2}(\alpha-1) t^4 - \frac{\alpha+2}{(\alpha-1)(\alpha-2)} t^2 + \frac{2\epsilon}{(\alpha-1)(\alpha-2)^2} & 
 \begin{tabular}{@{}l}
 \pI{}: $\alpha \neq 0,1,2,4$\\
 \pD{}: $\alpha = 4$
 \end{tabular}\\ \hline
 \sfT6 & (1,0,0) & \begin{array}{l}
 \begin{array}{@{}l}
 \begin{array}{@{}l}
 \bS = x_1 \partial_{x_1} + x_2 \partial_{x_2} + (\epsilon x_1+u) \partial_u,\\
 \bT = x_1 \partial_{x_2} + 2 x_2 \partial_u
 \end{array}
 \end{array} \\ \hline
 \begin{array}{@{}l}
 \be_1 = \bS, \quad
 \be_2 = \frac{\epsilon}{2} \bT,\\
 \bf_1 = \widetilde\bS := \bS - \partial_{x_1} - \epsilon\partial_u, \quad
 \bf_2 = \widetilde\bT := \bT - \partial_{x_2}
 \end{array}
 \end{array} & \frac{\epsilon}{2}(t^4 - 8 \epsilon t^2 + 4) & \pI{}\\ \hline
 \end{array}
 \]
 \caption{LC-adapted framings, quartics, and root types for some tubes}
 \label{F:tubular-LC-Q4}
\end{table}
\end{footnotesize}
 \vfill
  \end{landscape}
 \subsection{PDE point symmetries method}
 \label{S:PDE-method}
  
 In view of \eqref{E:sym-dim}, we may confirm $\dim\,\hol(M) = 5$ for the remaining type $\sfD$ and $\sfN$ tubular cases (from Table \ref{F:tubular-LC-Q4}) via their corresponding ILC structure (Table \ref{F:tube-PDE}).  In \S\ref{S:approach}, we described how to go from $M$ to this ILC structure realized as a PDE.  In this realization, the ILC symmetries are the {\sl point} symmetries of the PDE system \cite{Olver1995}.  There is excellent functionality in the {\tt DifferentialGeometry} package in {\tt Maple} for computing symmetries -- see below.
 
 \begin{footnotesize}
 \begin{table}[h]
 \[
 \begin{array}{|c|l|l|c|c|} \hline
 \mbox{Label} & \mbox{Real affine surface} & \mbox{Complete 2nd order PDE system}\\ \hline
 \sfT1 & u = \frac{1}{x_1 x_2} & \begin{cases} 
w_{11} = e^{2\pi i/3} w_1^{5/3} w_2^{-1/3}\\
w_{12} = \half e^{2\pi i/3} w_1^{2/3} w_2^{2/3}\\
w_{22} = e^{2\pi i/3} w_1^{-1/3} w_2^{5/3}\\ 
 \end{cases}
\\
 & u = x_1 x_2^\beta\quad(\beta \neq 0,\pm1) & \begin{cases}
 w_{11} = 0\\
 w_{12} = \frac{\beta}{2} w_1^{\frac{\beta-1}{\beta}}\\
 w_{22} = \frac{\beta-1}{2} w_2 w_1^{-\frac{1}{\beta}}
 \end{cases} \\ \hline
 \sfT2 & u = \frac{1}{x_1^2+x_2^2} & \begin{cases}
 w_{11} = \frac{2^{2/3}(3 w_1^2 - w_2^2)}{4(w_1^2+w_2^2)^{1/3}}\\
 w_{12} = \frac{2^{2/3} w_1 w_2}{(w_1^2+w_2^2)^{1/3}}\\
 w_{22} = \frac{2^{2/3}(3 w_2^2 - w_1^2)}{4(w_1^2+w_2^2)^{1/3}}
 \end{cases}\\
 & \begin{array}{@{}l} u = \exp(\beta \arctan(\frac{x_2}{x_1}))\\
 \quad (\beta \neq 0)
 \end{array} & 
 \begin{cases}
 w_{11} = (\frac{w_1^2}{2} - \frac{1}{\beta} w_1 w_2)\exp(\beta\arctan(\frac{w_1}{w_2}))\\
 w_{12} = \frac{1}{2} (\frac{1}{\beta}(w_1^2-w_2^2) + w_1 w_2)\exp(\beta\arctan(\frac{w_1}{w_2}))\\
 w_{22} = (\frac{w_2^2}{2} + \frac{1}{\beta} w_1 w_2)\exp(\beta\arctan(\frac{w_1}{w_2}))\\
 \end{cases}\\ \hline
 \sfT3 & u = x_1 \ln(x_2) & \begin{cases} 
w_{11} = 0 \\
w_{12} = \half e^{-w_1}\\
w_{22} = -\half w_2 e^{-w_1}
\end{cases}\\ \hline
 \sfT4 & 
 \begin{array}{@{}l}
 (u-x_1 x_2 + \frac{x_1^3}{3})^2 
 \, = \alpha(x_2 - \frac{x_1^2}{2})^3\\
 \quad (\alpha\neq 0, -\frac{8}{9})
 \end{array} & \begin{cases} 
w_{11} = -\frac{3\sqrt{\alpha(9\alpha+8)}(w_2^2+w_1)}{8\sqrt{w_2^2+2w_1}} - \frac{9\alpha+8}{8} w_2\\
w_{12} = \frac{3\sqrt{\alpha(9\alpha+8)} w_2}{16\sqrt{w_2^2+2w_1}}
+ \frac{9\alpha+8}{16}\\
w_{22} = -\frac{3\sqrt{\alpha(9\alpha+8)}}{16\sqrt{w_2^2+2 w_1}}
\end{cases} \\ \hline
 \sfT5 & x_1 u = x_2^2 + \epsilon x_1^4 & \begin{cases}
 w_{11} = \sqrt{3\epsilon} \frac{w_2^2+2w_1}{\sqrt{w_2^2+4w_1}}\\
 w_{12} = -\sqrt{3\epsilon} \frac{w_2}{\sqrt{w_2^2+4w_1}}\\
 w_{22} = 2\sqrt{3\epsilon}\frac{1}{\sqrt{w_2^2+4w_1}}
 \end{cases}\\ \hline
 \end{array}
 \]
 \caption{PDE realizations of some tubular ILC structures}
 \label{F:tube-PDE}
 \end{table}
 \end{footnotesize}

\begin{ex}[$\sfT3$, $\alpha=0$] The surface $u = x_1 \ln(x_2)$ has tube $M$ and complexification $M^c$:
\[
 M: \quad \Re(w) = \Re(z_1) \ln(\Re(z_2)), \qquad
 M^c: \quad \frac{w+c}{2} = \frac{z_1+a_1}{2} \ln\left(\frac{z_2+a_2}{2}\right).
\]
 For $M^c$, we solve for $w$ and differentiate twice:
 \begin{align}
 &(w_1,w_2,w_{11},w_{12},w_{22}) = \left( \ln\left( \frac{z_2+a_2}{2} \right),\, 
 \frac{z_1+a_1}{z_2+a_2},\, 0,\, \frac{1}{z_2+a_2},\, -\frac{z_1+a_1}{(z_2+a_2)^2} \right). \label{X:tube2}
 \end{align}
 Eliminating the parameters $(a_1,a_2,c)$ from \eqref{X:tube2}, we arrive at the PDE system given in Table \ref{F:tube-PDE}.  Using \eqref{E:PDE-EF}, we then confirm 5-dimensional symmetry via the following commands in {\tt Maple}:
\end{ex}
\begin{small}
 \begin{verbatim}
restart: with(DifferentialGeometry): with(GroupActions):
DGsetup([z1,z2,w,w1,w2],N):
w11:=0: w12:=1/2*exp(-w1): w22:=-1/2*w2*exp(-w1):
E:=evalDG([D_z1+w1*D_w+w11*D_w1+w12*D_w2,D_z2+w2*D_w+w12*D_w1+w22*D_w2]):
F:=evalDG([D_w1,D_w2]):
sym:=InfinitesimalSymmetriesOfGeometricObjectFields([E,F],output="list");
nops(sym);
 \end{verbatim}
 \end{small}

This similarly confirms the cases in Table \ref{F:tube-PDE} without parameters.  For the remaining cases with parameters, more care is needed since the above commands should at most be assumed to treat parameters {\em generically}.  To identify possible exceptional values, we should step-by-step solve the {\sl symmetry determining equations}.  Although we could set this up as infinitesimally preserving $E$ and $F$ as above, let us indicate another standard method.  Any {\sl point symmetry} $X$ is the prolongation $Y^{(1)}$ of a vector field $Y$ on $(z_1,z_2,w)$-space $J^0(\C^2,\C)$, and we can further prolong to get a vector field $Y^{(2)}$ on the second jet-space $J^2(\C^2,\C)$.  A PDE system is a submanifold $\Sigma \subset J^2(\C^2,\C)$, and the symmetry condition is that $Y^{(2)}|_\Sigma$ is everywhere tangent to $\Sigma$.  The following code efficiently sets this up in {\tt Maple} for the $\sfT1$ case $u = x_1 x_2^\beta$ for $\beta \neq 0,\pm 1$:

\begin{small}
\begin{verbatim}
restart: with(DifferentialGeometry): with(JetCalculus):
DGsetup([z1,z2],[w],J,2):
X:=evalDG(xi1(z1,z2,w[])*D_z1+xi2(z1,z2,w[])*D_z2+eta(z1,z2,w[])*D_w[]):
X2:=Prolong(X,2):
rel:=[w[1,1]=0,w[1,2]=beta/2*w[1]^((beta-1)/beta),
    w[2,2]=(beta-1)/2*w[2]*w[1]^(-1/beta)]:
eq:=eval(LieDerivative(X2,map(v->lhs(v)-rhs(v),rel)),rel):
\end{verbatim}
\end{small}

\noindent The expression {\tt eq} must vanish identically (for arbitrary $w_1,w_2$), and this gives a {\em highly overdetermined} system of linear PDE on the three coefficient functions $\xi_1,\xi_2,\eta$ of $Y$.  Keeping in mind $\beta \neq 0,\pm 1$, we solve these equations and confirm 5-dimensional symmetry.  Similar computations were carried out for the remaining parametric cases and the result was the same.  (For more details in the $\sfT1$ and $\sfT4$ cases, see the {\tt Maple} files accompanying the {\tt arXiv} submission of this article.)

 Family $\sfT2$ can be alternatively handled.  As remarked in \cite{AffHom}, the family of {\em complex} surfaces in $\C^3$ given by $u = x_1^\alpha x_2^\beta$ are $\Aff(3,\C)$-equivalent to surfaces in the $u = \left(x_1^2+x_2^2\right)^{\gamma}
\exp\left(\delta \arctan\Big(\frac{x_2}{x_1}\Big)\right)$ family.  (Here, $\alpha,\beta,\gamma,\delta \in \C$.)  Indeed, from their affine symmetry algebras, we deduce that they are $\Aff(3,\C)$-equivalent when 
 \begin{align} \label{E:T1T2-equiv}
(\alpha,\beta) = \left(\gamma + \tfrac{i}{2} \delta, \gamma - \tfrac{i}{2} \delta\right).
\end{align}  (One can also account for the `Redundancies' as in Table \ref{F:tubular-final}.)  By Remark \ref{R:CAff-hyp}, these complex surfaces yield tubular ILC structures and when \eqref{E:T1T2-equiv} holds, they are necessarily equivalent.  (A nice exercise derives the root types for $\sfT2$ from those of $\sfT1$ using \eqref{E:T1T2-equiv}.)  But now the remaining $\sfD$ and $\sfN$ cases for $\sfT2$ are equivalent to the $\sfD$ and $\sfN$ cases for $\sfT1$, which were already treated, and so we are done.
 
 \subsection{Power series method}
 \label{S:ps-method}
 In this section, we outline a second method for the algorithmic computation of the infinitesimal symmetries of tubular CR hypersurfaces (or rather tubular ILC structures).  We express this in the language of elementary linear algebra.

 \subsubsection{Filtered linear equations}\label{ss:flineq}
Let $V$ be a {\sl filtered} vector space, i.e.\
 \[
 V =: V^{\mu_0} \supset V^{\mu_0+1} \supset V^{\mu_0+2} \supset \ldots, \qquad \bigcap_\mu V^\mu = 0.
 \]
 Let $\gr V := \bigoplus_\mu V^\mu / V^{\mu+1}$ be its {\sl associated graded} vector space.  Any subspace $W \subset V$ inherits a filtration from $V$, and note that $\dim \gr W = \dim\,W $.

Let $U$ be another filtered vector space and $\phi\colon V \to U$ a {\sl filtration-homogeneous linear map of degree $k$}, i.e.\ $\phi(V^\mu)\subset U^{\mu+k}$ for all $\mu\in\bbZ$. Denote by $\gr \phi\colon \gr V\to \gr U$ the corresponding graded map (of degree $k$).
In applications, we often know the map $\gr\phi$ and its kernel $\ker\gr \phi$, and would like to use this information in order to determine $\ker\phi$. 

\begin{lem}\label{lem1} $\gr\ker\phi\subset \ker\gr\phi$.
\end{lem}

\begin{proof}
 Let $v\in \ker\phi$. Let $\mu$ be the largest integer such that $v\in V^\mu$. Then $\phi(v)\in U^{\mu+k}$ and $(\gr \phi) (v+ V^{\mu+1}) = \phi(v)+U^{\mu+k+1} = 0$, and thus $v+V^{\mu+k+1}\in \ker\gr\phi$. 
\end{proof}
The inclusion in Lemma~\ref{lem1} can be strict, so $\dim \ker \gr \phi$ is only an upper bound for $\dim \ker\phi = \dim \gr \ker\phi$.  
 \subsubsection{Symmetry equations as filtered linear equations}
Given a real hypersurface $M \subset \C^3$, its complexification is a complex hypersurface $M^c \subset \C^3 \times \C^3$ graphed as\footnote{In this section, we use the complex variables $(x,y,z,a,b,c)$ instead of $(z_1,z_2,w,a_1,a_2,c)$.}:
\begin{align}
	M^c
	\colon\ \ \ \ \
	z
	\,=\,
	Q(x,y,a,b,c),
\end{align}
with $Q$ analytic, \emph{i.e.} expandable in a 
converging power series. We may assume $0 \in M^c$, \emph{i.e.} $0 = Q(0, 0, 0, 0, 0)$.
We consider $M^c$ up to the pseudogroup of local analytic
transformations:
\begin{multline}
	\label{transformations-xyz-abc}
	(x,y,z,a,b,c)
	\mapsto 
	\big(
	x'(x,y,z),y'(x,y,z),z'(x,y,z),
	a'(a,b,c),b'(a,b,c),c'(a,b,c)
	\big).
\end{multline}
The Lie algebra $\sym(M^c)$ of infinitesimals symmetries consists of those vector fields
\begin{align}\label{vf-xyz-abc}
	\begin{split}
		L &\,=\, 
		X(x,y,z)\,\partial_x
		+
		Y(x,y,z)\,\partial_y
		+
		Z(x,y,z)\,\partial_z
		\\
		&\qquad +
		A(a,b,c)\,\partial_a
		+
		B(a,b,c)\,\partial_b
		+
		C(a,b,c)\,\partial_c
	\end{split}
\end{align}
that are tangent to $M^c$.  We will make the assumption that $M^c$ is \textsl{rigid:}
\begin{align}
	z
	\,=\,
	-\,c
	+
	F(x,y,a,b),
\end{align}
with $0 = F(0,0,0,0)$.  ({\sl Tubes} form the subclass $z = -c + F(x+y,a+b)$.)  The rigidity assumption is justified when $M^c$ is homogeneous, whence there exists at least one $L \in \sym(M^c)$ with $L(0)$ not tangent to the $4$-dimensional 
contact distribution. After a straightening, one can make $L = \partial_z - \partial_c$, and tangency to $\{ z = Q\}$ forces $Q = - c + F$ as above.

 \begin{remark}
Up to the transformations~\eqref{transformations-xyz-abc}, we can assume that $F$ does not contain constant or linear terms in $x,y,a,b$. Specifying second order terms, we get:
\begin{align}\label{eq-def-G}
	z
	\,=\,
	-\,c + \ell(x,y,a,b)
	+
	G(x,y,a,b),
\end{align}
with quadratic term $\ell(x,y,a,b) = e\,xa+f\,xb+g\,ya+h\,yb$ for $e,f,g,h \in \C$
satisfying
$0 \neq \big\vert \begin{smallmatrix} e & f \\ g & h 
\end{smallmatrix}\big\vert$ by Levi non-degeneracy of the original hypersurface $M \subset \bbC^3$, and $G$ containing higher order terms in $x,y,a,b$. Using linear transformations of $(x,y)$ and $(a,b)$, we can assume that $\ell(x,y,a,b) = xa+yb$.

 \end{remark}
Now, express the tangency condition as:
\begin{align}
	\aligned
	0
	&
	\,\equiv\, \eqdef_F(L) := 
	L
	\big(
	-z-c+F(x,y,a,b)
	\big)
	\Big\vert_{z=-c+F}
	\\
	&
	\,\equiv\,
	\big[
	X\,F_x
	+
	Y\,F_y
	-
	Z
	+
	A\,F_a
	+
	B\,F_b
	-
	C
	\big]
	\Big\vert_{z=-c+F},
	\endaligned
\end{align}
which reads as the identical vanishing of the following
power series in $5$ variables $(x,y,a,b,c)$:
\begin{align}
	\aligned
	0
	&
	\,\equiv\,
	X\big(x,y,-c+F(x,y,a,b)\big)\,
	F_x(x,y,a,b)
	\\
	&
	\ \ \ \ \
	+
	Y\big(x,y,-c+F(x,y,a,b)\big)\,
	F_y(x,y,a,b)
	-
	Z\big(x,y,-c+F(x,y,a,b)\big)
	\\
	&
	\ \ \ \ \
	+
	A(a,b,c)\,F_a(x,y,a,b)
	+
	B(a,b,c)\,F_b(x,y,a,b)
	-
	C(a,b,c).
	\endaligned
\end{align}

 Now $\phi(L) = \eqdef_F(L)$ defines a linear map $\phi 
 : V \to U$ from the Lie algebra $V$ of all analytic vector fields~\eqref{vf-xyz-abc} to the space $U$ of all analytic functions in $(x,y,a,b,c)$. Then we have
\begin{align}
	\sym(M^c) = \ker \phi.
\end{align}
Expanding $\phi(L)$ in a power series and evaluating the coefficients of this series degree by degree, we can view the computation of $\ker \phi$ as an (infinite) system of linear equations on the coefficients of the power series expansion of $L$, where the coefficients of these linear equations are formed by some algebraic expressions of the power series coefficients of $F$. 

We now endow $V$ and $U$ with filtrations.  Assigns weights $(1,1,2,1,1,2)$ to $(x,y,z,a,b,c)$, and $(-1,-1,-2,-1,-1,-2)$ to $(\partial_x, \partial_y, \partial_z, \partial_a, \partial_b, \partial_c)$.  Define $V^\mu \subset V$ and $U^\mu \subset U$ as the weight $\geq \mu$ subspaces.  (Note that $V = V^{-2}$, while $U = U^0$.)  Then $\phi$ is filtration-homogeneous and restricts to $\phi : V^\mu \to U^{\mu+2}$, i.e.\ it has degree $+2$.

The associated graded spaces $\gr V$ and $\gr U$ can be identified with polynomial vector fields of the form~\eqref{vf-xyz-abc} and polynomials in $(x,y,a,b,c)$ respectively.  An elementary computation shows that
\begin{align}
\gr \phi = \eqdef_{\ell},
\end{align}
where the right hand side defines the equations for the infinitesimal symmetries of the flat model $\{z = -c + \ell \}$, which is defined by a homogeneous equation of weight $2$. 

The symmetry algebra of the flat model is well-known to be the 15-dimensional Lie algebra of polynomial vector fields having dimensions $(1,4,5,4,1)$ in degrees $(-2,-1,0,1,2)$ respectively.  From Lemma~\ref{lem1}, we immediately recover the well-known fact that $\dim\,\sym(M^c) \leq 15$, and each symmetry $L$ is uniquely determined by its terms of weight $\le 2$.

Our aim is to use knowledge of $\ker \eqdef_{\ell}$ to effectively compute $\ker \phi$.  Fixing an integer parameter $\nu$, define the following {\em finite-dimensional} quotient vector spaces:
\begin{align*}
 V(\nu) = V  \mod V^{\nu+1}, \quad
 U(\nu) = U \mod U^{\nu+1},
\end{align*}
which inherit filtrations from $V$ and $U$.  Now $\phi$ induces a filtration-homogeneous map of degree +2:
\begin{align}
 \begin{split}
 \phi(\nu) : V(\nu) &\longrightarrow U(\nu+2)\\
 [L] &\longmapsto [\eqdef_{F}(L)],
 \end{split}
\end{align}
where brackets denote the respective equivalence classes.
Then $\ker \phi(\nu)$ approximates $\sym(M^c) = \ker(\phi)$ modulo terms of weight $\ge\nu+1$. For increasing $\nu$, we have that $\dim\,\ker\phi(\nu)$ is a decreasing sequence of integers stabilizing at $\dim \sym(M^c)$. 

\begin{remark} For the tubes in Table \ref{F:tubular-final}, this sequence stabilizes already for $\nu=4$.
\end{remark}
 \subsubsection{Symmetry computation}
Fix $\nu$, and for ease of exposition in this subsection, set $\phi:=\phi(\nu)$, $V:=V(\nu)$, $U:=U(\nu+2)$. The following is an effective algorithm for computing $\ker\phi$ based on the knowledge of $\gr\phi$:
\begin{enumerate}
	\item Find $\ker\gr\phi\subset \gr V$;
	\item Choose a subspace $\mathring V\subset V$ with
	$\gr V = \gr \mathring V \oplus \ker\gr\phi$.  (This means that $\gr\phi$ is injective on $\gr \mathring V$.  By Lemma~\ref{lem1}, $\phi$ is also injective on $\mathring V$.)	
	\item Compute $\gr(\phi(\mathring V)) = (\gr\phi)(\gr\mathring V)$. Choose a subspace $\mathring U\subset U$ with $\gr U = (\gr\phi)(\gr\mathring V)\oplus \gr\mathring U$, so that the induced maps $\mathring{V} \to U/\mathring U$ and $\gr \mathring{V} \to \gr U / \gr \mathring{U}$ are isomorphisms.  Thus:
	\begin{align}
 \begin{array}{cccc} & \gr V & \stackrel{\gr\phi}{\xrightarrow{\hspace*{2cm}}} &  \gr U\\
 & \verteq & & \verteq\\[-0.15in]
 & \gr \mathring{V} & \stackrel{\gr\phi|_{\gr \mathring{V}}}{\xrightarrow{\hspace*{2cm}}} & (\gr \phi)(\gr \mathring{V})\\
 & \oplus & & \oplus\\
 \gr \ker\phi \subset & \ker \gr \phi & & \gr \mathring{U}
 \end{array} \quad\quad
 \begin{array}{cccc}
 V & \stackrel{\phi}{\xrightarrow{\hspace*{2cm}}} & U\\
 \rotatebox{90}{$\subset$} &&\downarrow\\[-0.1in]
 \mathring{V} & \underset{\cong}{\stackrel{\mathring\phi}{\xrightarrow{\hspace*{2cm}}}} & U / \mathring{U}\\\\\\
 \end{array}
 \end{align}
	\item Consider the map $\mathring\phi\colon V\to U/\mathring U$. By what precedes, $\ker\mathring\phi$ has the same dimension as $\ker\gr\mathring \phi$, is complementary to $\mathring V$ and contains $\ker\phi$.
	\item Finally, consider the map $\phi\colon \ker\mathring\phi \to \mathring U$ and compute its kernel.
\end{enumerate}	

The key computational advantage of this approach is that the first four steps do not involve any parameter dependency introduced by $G$ in~\eqref{eq-def-G}.  This allows one to reduce the parametric analysis for $\dim\,\sym(M^c)$ to the last step in the above algorithm.  Let us describe this in more detail.
	
Choose bases (consisting of homogeneous elements) of $\gr V = \gr \mathring V \oplus \ker \gr \phi$ and $\gr U= (\gr\phi)(\gr \mathring V)\oplus \gr \mathring U$ adapted to the given decompositions.  Then extend these bases to $V$ and $U$ in a manner compatible with the choices of subspace $\mathring V \subset V$ and $\mathring U \subset U$.  In these bases,
\begin{align}
 \gr \phi = \begin{pmatrix} A & 0\\ 0 & 0 \end{pmatrix}, \quad
 \phi = \begin{pmatrix} B_{11} & B_{12} \\ B_{21} & B_{22} \end{pmatrix}.
\end{align}
 Here, $A$ and $B_{11}$ are non-degenerate and correspond to the isomorphisms $\gr \mathring V\to \gr U/\gr\mathring U$ and $\mathring V \to U/\mathring U$ respectively.  Moreover, $\gr B_{11} = A$ by construction, so it does not depend on the function $G$ in the defining equation~\eqref{eq-def-G} for $M^c$. This means that computation of the kernel $\mathring\phi\colon V\to U/\mathring U$ does not introduce any dependency on the parameters that may appear in $G$. Thus, the dependency of $\dim\ker\phi$ on $G$ appears only on step~(5), which significantly reduces the computational complexity.
	
By a careful choice of the subspaces $\mathring V$ and $\mathring U$, we reduce computation of $\ker\phi(\nu)$ for $\nu=2,3,4$ to systems of $5$, $25$, $75$ linear equations respectively on $\dim \ker\gr\phi(\nu) = 15$ variables.  (For sample details in the $\sfT4$ case, see the {\tt Maple} files supplementing the {\tt arXiv} submission of this article.)  We note that the direct analysis of the corank of the map $\phi(4) : V(4) \to U(6)$ without applying the techniques of filtered linear equations would result in dealing with $\dim\, U(6) = 130$ linear equations in $\dim\, V(4) = 80$ variables.
 \subsection{Conclusion}
 As described in \S\ref{S:PDE-method} and \S\ref{S:ps-method}, we used two different methods to confirm:

\begin{prop} \label{P:tubes-almost-done}
 Any tubular hypersurface $M^5 \subset \C^3$ from Table \ref{F:tubular-final} has $\dim\,\hol(M) = 5$.
\end{prop}

Finally, we address whether there is any redundancy in our (tubular) list.  The following slightly weakens the `uniqueness' hypothesis from \cite[Prop.4.1]{KL2019}.  (The proof is the same.)
 
 \begin{prop} \label{P:base-equiv} Let $M_1, M_2 \subset \C^{n+1}$ be two tubular hypersurfaces over affinely homogeneous bases $\cS_1, \cS_2 \subset \R^{n+1}$.  Suppose that $M_1$ and $M_2$ are holomorphically simply-transitive and that $\langle i\partial_{z_1},\ldots,i\partial_{z_{n+1}} \rangle$ is a {\em characteristic}\footnote{An ideal in a Lie algebra is {\em characteristic} if it is preserved by all automorphisms of the Lie algebra.} $(n+1)$-dimensional abelian ideal in $\hol(M_1)$ and $\hol(M_2)$.  Then $M_1$ and $M_2$ are locally biholomorphically equivalent if and only if their bases are locally affinely equivalent.
 \end{prop}

 We confirm the characteristic property via corresponding ILC data $(\fg;\fe,\ff)$ and $(\fa,\tau)$:
 \begin{itemize}
 \item $\sfT1,\sfT2,\sfT3,\sfT6$: $\fa$ is the derived algebra of $\fg$.
 \item $\sfT4,\sfT5$: $\fa$ is the centralizer of the (2-dimensional) second derived algebra of $\fg$.
 \end{itemize}
 This implies that $\fa \subset \fg$ is characteristic, so the corresponding abelian ideal in $\hol(M)$ is characteristic, and hence Proposition \ref{P:base-equiv} applies.  From the DKR classification \cite{AffHom}, there is no affine equivalence between $\cS_1$ and $\cS_2$ lying in different families among $\sfT1$--$\sfT6$.  For $\cS_1$ and $\cS_2$ in the same family, we can assess affine equivalence by asking if $\aff(\cS_1)$ and $\aff(\cS_2)$ are conjugate in $\mathfrak{aff}(3,\R)$.  We leave this as a straightforward exercise for the reader.  This gives rise to the `{\sl Redundancy}' conditions in Table \ref{F:tubular-final}, e.g.\ in $\sfT1$, $(\alpha, \beta) \sim (\frac{1}{\alpha}, -\frac{\beta}{\alpha})$ is induced from the swap $(x_1,x_2,u) \mapsto (u,x_2,x_1)$.\\

The proof of Theorem \ref{thm-main} is now complete.
 \section*{Acknowledgements}
 \thispagestyle{empty}

The authors acknowledge the use of the DifferentialGeometry package in Maple. The research leading to these results has received funding from the Norwegian Financial Mechanism 2014-2021 (project registration number 2019/34/H/ST1/00636), the Polish National Science Centre (NCN) (grant number 2018/29/B/ST1/02583), and the Troms\o{} Research Foundation (project ``Pure Mathematics in Norway'').

\end{document}